\documentclass[a4paper]{amsart}

\usepackage[T1]{fontenc}
\usepackage[utf8]{inputenc}
\usepackage[english]{babel}
\usepackage{amssymb,amsmath,amsthm,amscd}

\usepackage{caption}
\usepackage{subcaption}
\usepackage{url}
\usepackage[all]{xy}

\newcommand{\CP}{\mathbb{C}\mathbb{P}^{1}}
\newcommand{\RP}{\mathbb{R}\mathbb{P}^{1}}
\newcommand{\RPP}{\mathbb{R}\mathbb{P}^{2}}
\newcommand{\CPP}{\mathbb{C}\mathbb{P}^{2}}

\newcommand{\R}{\mathbb{R}}
\newcommand{\C}{\mathbb{C}}
\newcommand{\Z}{\mathbb{Z}}
\newcommand{\ZZ}{\mathbb{Z}/2\mathbb{Z}}
\newcommand{\N}{\mathbb{N}}

\newcommand{\lra}{\longrightarrow}
\newcommand{\lmt}{\longmapsto}

\newcommand{\Si}{\Sigma}

\renewcommand{\Im}[1]{\mathrm{Im}\, #1}

\newtheorem{thm}{Theorem}[section]
\newtheorem{lm}[thm]{Lemma}
\newtheorem{coro}[thm]{Corollary}
\newtheorem{prop}[thm]{Proposition}
\theoremstyle{definition}
\newtheorem{df}[thm]{Definition}
\theoremstyle{plain}

\newcommand{\tv}{$\times$-vertex}
\newcommand{\tvs}{$\times$-vertices}

\DeclareMathOperator{\Cut}{Cut}
\DeclareMathOperator{\Ver}{Ver}
\DeclareMathOperator{\Ind}{Ind}
\DeclareMathOperator{\Dssn}{Dssn}

\DeclareMathOperator{\Sing}{Sing}

\usepackage{graphicx}
\usepackage{hyperref}\usepackage{tikz}
\usetikzlibrary{arrows,decorations.pathmorphing,decorations.pathreplacing,backgrounds,positioning,fit,petri,calc,cd}

\title{Generic pointed quartic curves in $\RPP$ and uninodal dessins}
\author{Andrés Jaramillo Puentes}

\begin{document}
\begin{abstract}

In this article we obtain a rigid isotopy classification of generic pointed quartic curves $(A,p)$ in $\RPP$ by studying the combinatorial properties of dessins.
The dessins are real versions, proposed by S. Orevkov~\cite{Orevkov}, of Grothendieck's {\it dessins d'enfants}.

This classification contains 20 classes determined by the number of ovals of~$A$, the parity of the oval containing the marked point $p$, the number of ovals that the tangent line~$T_p A$ intersects, the nature of connected components of~$A\setminus T_p A$ adjacent to $p$, and in the maximal case, on the convexity of the position of the connected components of~$A\setminus T_p A$.

We study the combinatorial properties and decompositions of dessins corresponding to real uninodal trigonal curves in real ruled surfaces.
Uninodal dessins in any surface with non-empty boundary can be decomposed in blocks corresponding to cubic dessins in the disk~$\mathbf{D}^2$, which produces a classification of these dessins.
The classification of dessins under consideration leads to a rigid isotopy classification of generic pointed quartic curves in $\RPP$.
This classification was first obtained in~\cite{Rieken} based on the relation between quartic curves and del Pezzo surfaces.
\end{abstract}
\maketitle
\tableofcontents
\section*{Introduction}

The first part of Hilbert's 16th problem \cite{Hilbert} asks for a topological classification of all possible pairs $(\RPP,\R A)$,
where $\R A$ is the real point set of a non-singular curve~$A$ of fixed degree $d$ in the real projective plane $\RPP$. 

The fact that any homeomorphism of $\RPP$ is isotopic to the identity implies that two topological pairs $(\RPP,\R A)$ and $(\RPP,\R B)$ are homeomorphic if and only if~$\R A$ is isotopic to $\R B$ as subsets of $\RPP$. 

The moduli space $\mathbb{RP}^{\frac{d(d+3)}{2}}$ of real homogeneous polynomials of degree $d$ in three variables (up to multiplication by a non-zero real number) has open strata formed by the non-singular curves (defined by polynomials with non-vanishing gradient in $\C^3\setminus\{\bar{0}\}$) and a codimension one subset formed by the singular curves; the latter subset is an algebraic hypersurface called the \emph{discriminant}. 
The discriminant divides $\mathbb{RP}^{\frac{d(d+3)}{2}}$ into connected components, called {\it chambers}.
Two curves in the same chamber can be connected by a path that does not cross the discriminant, 
and therefore every point of the path corresponds to a non-singular curve. Such a path is called a \emph{rigid isotopy}. 
A version of Hilbert's 16th problem asks for a description of the set of chambers of the moduli space $\mathbb{RP}^{\frac{d(d+3)}{2}}$, equivalently, for a classification up to rigid isotopy of non-singular curves of degree~$d$ in $\RPP$. 
More generally, given a deformation class of complex algebraic varieties, one can ask for a description of the set of real varieties up to equivariant deformation within this class.
In the case of marked curves $(C,p)$ in $\RPP$, the open strata correspond to non-singular curves $C$ such that the tangent line $T_p C$ intersects $C\setminus\{p\}$ transversally.

We use a correspondence between generic marked quartic curves $(C,p)$ in $\RPP$ and dessins in $\mathcal{D}^2$ with only a singular vertex, which is a nodal vertex.
The above correspondence makes use of trigonal curves on the Hirzebruch surface $\Sigma_2$. In this text, a trigonal curve is a complex curve $C$ lying in a ruled surface such that the restriction to $C$ of the projection provided by the ruling is a morphism of degree $3$. 
We consider an auxiliary morphism of $j$-invariant type in order to associate to the curve $C$ a dessin, 
a real version, proposed by S. Orevkov, of the dessins d'enfants introduced by A. Grothendieck.

In Section \ref{ch:tri}, we introduce the basic notions and discuss some
properties of trigonal curves and the associated dessins in the complex and real case.
We explain how the study of equivalence classes of dessins allows us to obtain information about the equivariant deformation classes of real trigonal curves. 

In Section \ref{ch:uni}, we proved that a dessin lying on an arbitrary compact surface with boundary and having only one nodal vertex is weakly equivalent to a dessin having a decomposition as gluing of dessins corresponding to cubic curves in~$\RPP$.
These combinatorial properties are valid in a more general setting than the one needed for the case of pointed quartic curves in $\RPP$.

In Section \ref{ch:hir}, we deal with a relation between plane curves and trigonal curves on the Hirzebruch surfaces and detail the case of plane cubic curves, which leads to the dessins that serve as building blocks for our constructions.

An application of Theorem~\ref{prop:unide} to the case of uninodal dessins representing real pointed quartic curves lead us to a combinatorial classification of these dessins. 

The rigid isotopy classification classification of generic real pointed quartic was first obtained in~\cite{Rieken} based on the relation between quartic curves and del Pezzo surfaces.
Lastly, in Theorem~\ref{thm:quart} we explicit this rigid isotopy classification, represented by the figures in Tables~\ref{fig:d4b01} to \ref{fig:d4b04}. This classification contains 20 classes determined by the number of ovals of the curve~$C$, the parity of the oval containing the marked point~$p$, the number of ovals that the tangent line~$T_p C$ intersects, the number of connected components of~$C\setminus T_p C$, and in the maximal case, on the convexity of the position of the connected components of~$C\setminus T_p C$.

\section{Trigonal curves and dessins}\label{ch:tri}

In this chapter we introduce trigonal curves and dessins, which are the principal tool we use in order to study the rigid isotopy classification of generic pointed smooth curves of degree~$4$ in~$\RPP$.
The content of this chapter is based on the book~\cite{deg} and the article~\cite{DIZ}.

\subsection{Ruled surfaces and trigonal curves}

\subsubsection{Basic definitions}
A compact complex surface $\Sigma$ is a \emph{(geometrically) ruled surface} over a curve $B$ 
if $\Si$ is endowed with a projection $\pi:\Si\lra B$ of fiber~$\CP$ as well as a special section $E$ 
of non-positive self-intersection.

\begin{df}
A \emph{trigonal curve} is a reduced curve $C$ lying in a ruled surface $\Sigma$ such that $C$ contains neither the exceptional section $E$ nor a fiber 
 as component, and	the restriction $\pi|_{C}:C\lra B$ is a degree $3$ map.
 
A trigonal curve $C \subset \Sigma$ is {\it proper} if it does not intersect the exceptional section~$E$.
A {\it singular fiber} of a trigonal curve $C\subset\Si$ is a fiber~$F$ of $\Si$ intersecting $C\cup E$ geometrically in less than $4$ points.
\end{df} 

A fiber $F$ is singular if $C$ passes through $E\cap F$, or if~$C$ is tangent to $F$ or if $C$ has a singular point belonging to $F$ (those cases are not exclusive). A singular fiber~$F$ is {\it proper} if $C$ does not pass through $E\cap F$. Then, $C$ is proper if and only if all its singular fibers are.
We call a singular fiber~$F$ \emph{simple} if either $C$ is tangent to $F$ or $F$ contains a node (but none of the branches of $C$ is tangent to $F$), a cusp or an inflection point. The set $\{b\in B\mid F_b \text{ is a singular fiber }\}$ of points in the base having singular fibers is a discrete subset of the base $B$. We denote its complement by $B^{\#}=B^{\#}(C)$; it is a curve with punctures.
We denote by $F_b^0$ the complement $F_b\setminus\{F_b\cap E\}$.

\subsubsection{Deformations} 
We are interested in the study of trigonal curves up to deformation. In the real case, we consider the curves up to equivariant deformation (with respect to the action of the complex conjugation, cf.~\ref{sssec:realstruc}).

In the Kodaira-Spencer sense, a \emph{deformation} of the quintuple $(\pi\colon\Si\lra B,E,C)$ refers to an analytic space $X\lra S$ fibered over an marked open disk $S\ni o$ endowed with analytic subspaces $\mathcal{B}, \mathcal{E}, \mathcal{C}\subset X$ such that for every $s\in S$, the fiber~$X_{s}$ is diffeomorphic to $\Si$ and the intersections $\mathcal{B}_s:= X_s\cap \mathcal{B}$, $\mathcal{E}_s:= X_s\cap \mathcal{E}$ and $\mathcal{C}_s:= X_s\cap \mathcal{C}$ are diffeomorphic to $B$, $E$ and $C$, respectively, and there exists a map $\pi_s\colon X_s\lra B_s$ making $X_s$ a geometrically ruled surface over $B_s$ with exceptional section $E_s$, such that the diagram in Figure~\ref{fig:KodSpe} commutes
and $(\pi_o\colon X_o\lra B_o,E_o,C_o)=(\pi\colon\Si\lra B,E,C)$.

\begin{figure}[h]
\begin{center}
\begin{tikzcd}
	& E_s \arrow[r, "\text{diff.}"] \arrow[d, hook] & E \arrow[d, hook] & \\
C_s \arrow[rd, "\pi_s|_{C_s}"'] \arrow[r, hook]	& X_s \arrow[r, "\text{diff.}"] \arrow[d, "\pi_s"] & \Sigma	\arrow[d,"\pi"'] \arrow[r,hookleftarrow] & C \arrow[ld, "\pi|_C"]  \\
	& B_s \arrow[r, "\text{diff.}"] & B &
\end{tikzcd}
\end{center}
\caption{Commuting diagram of morphisms and diffeomorphisms of the fibers of a deformation.}
\label{fig:KodSpe}
\end{figure}

\begin{df}
 An \emph{elementary deformation} of a trigonal curve $C\subset\Si\lra B$ is a deformation of the quintuple $(\pi\colon\Si\lra B,E,C)$ in the Kodaira-Spencer sense.
\end{df}

An elementary deformation $X\lra S$ is \emph{equisingular} if for every $s\in S$ there exists a neighborhood $U_s\subset S$ of $s$ such that for every singular fiber~$F$ of $C$, there exists a neighborhood $V_{\pi(F)}\subset B$ of $\pi(F)$, where $\pi(F)$ is the only point with a singular fiber for every $t\in U_s$.
An elementary deformation over $D^2$ is a \emph{degeneration} or \emph{perturbation} if the restriction to $D^2\setminus\{0\}$ is equisingular and for a set of singular fibers~$F_i$ there exists a neighborhood $V_{\pi(F_i)}\subset B$ where there are no points with a singular fiber for every $t\neq0$. In this case we say that $C_t$ \emph{degenerates} to $C_0$ or $C_0$ is \emph{perturbed} to $C_t$, for $t\neq0$.

\subsubsection{Nagata transformations} 
One of our principal tools in this paper are the dessins, which we are going to associate to proper trigonal curves. A trigonal curve $C$ intersects the exceptional section $E$
in a finite number of points, since 
$C$ does not contain 
$E$ as component. We use the Nagata transformations in order to construct a proper trigonal curve out of a non-proper one. 

\begin{df}
 A {\it Nagata transformation} is a fiberwise birational transformation $\Si\lra\Si'$ consisting of blowing up a point $p\in\Si$ (with exceptional divisor $E''$) and contracting the strict transform of the fiber~$F_{\pi(p)}$ containing $p$. 
The new exceptional divisor $E'\subset\Si'$ is the strict transform of $E\cup E''$.
 The transformation is called {\it positive} if $p\in E$, and {\it negative} otherwise.
\end{df}
 
The result of a positive Nagata transformation is a ruled surface $\Si'$, with an exceptional divisor $E'$ such that $-E'^2=-E^2+1$. The trigonal curves $C_1$ and $C_2$ over the same base $B$ are {\it Nagata equivalent} if there exists a sequence of Nagata transformations mapping one curve to the other by strict transforms. Since all the points at the intersection $C\cap E$ can be resolved, every trigonal curve $C$ is Nagata equivalent to a proper trigonal curve $C'$ over the same base, called a {\it proper model} of $C$.
 
 \subsubsection{Weierstraß equations}
 For a trigonal curve, the Weierstraß equations are an algebraic tool which allows us to study the behavior of the trigonal curve with respect to the zero section and the exceptional one. They give rise to an auxiliary morphism of $j$-invariant type, which plays an intermediary role between trigonal curves and dessins. 
Let $C\subset\Si\lra B$ be a proper trigonal curve. Mapping a point $b\in B$ of the base to the barycenter of the points in $C\cap F_B^0$ (weighted according to their multiplicity) defines a section $B\lra Z\subset\Si$ called the {\it zero section}; it is disjoint from the exceptional section $E$.\\

The surface $\Si$ can be seen as the projectivization of a rank $2$ vector bundle, which splits as a direct sum of two line bundles such that the zero section $Z$ corresponds to the projectivization of $\mathcal{Y}$, one of the terms of this decomposition. In this context, the trigonal curve $C$ can be described by a Weierstraß equation, which in suitable affine charts has the form
\begin{equation} \label{eq:Weiertrass}
x^3+g_{2}x+g_{3}=0,
\end{equation}
where $g_{2}$, $g_{3}$ are sections of $\mathcal{Y}^2$, $\mathcal{Y}^3$ respectively, and $x$ is an affine coordinate such that $Z=\{x=0\}$ and $E=\{x=\infty\}$. For this construction, we can identify $\Si\setminus B$ with the total space of $\mathcal{Y}$ and take $x$ as a local trivialization of this bundle. Nonetheless, the sections $g_{2}$, $g_{3}$ are globally defined. The line bundle $\mathcal{Y}$ is determined by $C$. The sections $g_{2}$, $g_{3}$ are determined up to change of variable defined by 
\begin{equation*}
(g_{2}, g_{3})\lra(s^2g_{2},s^3 g_{3} ),\;s\in H^{0}(B,\mathcal{O}^{*}_{B}).
\end{equation*}
Hence, the singular fibers of the trigonal curve $C$ correspond to the points where the equation~\eqref{eq:Weiertrass} has multiple roots, i.e., the zeros of the discriminant section
\begin{equation}
 \Delta:=-4g_2^3-27g_3^2\in H^0(B,\mathcal{O}_B(\mathcal{Y}^6)).
\end{equation}
Therefore, $C$ being reduced is equivalent to $\Delta$ being identically zero.
A Nagata transformation over a point $b\in B$ changes the line bundle $\mathcal{Y}$ to $\mathcal{Y}\otimes\mathcal{O}_B(b)$ and the sections $g_2$ and $g_3$ to $s^2 g_2$ and $s^3 g_3$, where $s\in H^0(B,\mathcal{O}_B)$ is any holomorphic function having a zero at $b$.

\begin{df} \label{df:gen}
 Let $C$ be a non-singular trigonal curve with Weierstraß model determined by the sections $g_2$ and $g_3$ as in \eqref{eq:Weiertrass}. The trigonal curve $C$ is {\it almost generic} if every singular fiber corresponds to a simple root of the determinant section $\Delta=-4g_2^3-27g_3^2$ which is not a root of $g_2$ nor of $g_3$. The trigonal curve $C$ is {\it generic} if it is almost generic and the sections $g_2$ and $g_3$ have only simple roots. 
\end{df}

\subsubsection{The $j$-invariant}
The $j$-invariant describes the relative position of four points in the complex projective line $\CP$. We describe some properties of the $j$-invariant in order to use them in the description of the dessins.

\begin{df}
 Let $z_1$, $z_2$, $z_3$, $z_4\in\CP$. The $j$-{\it invariant} of a set $\{z_1, z_2, z_3, z_4\}$ is given by
\begin{equation} \label{eq:jinvariant}
\displaystyle j(z_1, z_2, z_3, z_4)=\frac{4(\lambda^2-\lambda+1)^3}{27\lambda^2(\lambda-1)^2},
\end{equation}
where $\lambda$ is the {\it cross-ratio} of the quadruple $(z_1, z_2, z_3, z_4)$ defined as

 \begin{equation*}
\displaystyle\lambda(z_1, z_2, z_3, z_4)=\frac{z_1-z_3}{z_2-z_3} : \frac{z_1-z_4}{z_2-z_4}.
\end{equation*}
 \end{df}
 
The cross-ratio depends on the order of the points while the $j$-invariant does not. Since the cross-ratio $\lambda$ is invariant under Möbius transformations, so is the $j$-invariant. When two points $z_i$, $z_j$ coincide, the cross-ratio $\lambda$ equals either $0$, $1$ or $\infty$, and the $j$-invariant equals $\infty$. 
 
Let us consider 
a polynomial $z^3+g_2z+g_3$. We define the $j$-invariant $j(z_1, z_2, z_3)$ of its roots $z_1$, $z_2$, $z_3$ as $j(z_1, z_2, z_3,\infty)$. If $\Delta=-4g_2^3-27g_3^2$ is the discriminant of the polynomial, then
\begin{equation*}
\displaystyle j(z_1, z_2, z_3,\infty)=\frac{-4g_2^3}{\Delta}.
\end{equation*}

A subset $A$ of $\CP$ is real if $A$ is invariant under the complex conjugation. We say that $A$ has a nontrivial symmetry if there is a nontrivial permutation of its elements which extends to a linear map $z\lmt az+b$, $a\in\C^*$, $b\in\C$.

\begin{lm}[\cite{deg}]
The set $ \{z_1, z_2, z_3\}$ of roots of the polynomial $z^3+g_2z+g_3$ has a nontrivial symmetry if and only if its $j$-invariant equals $0$ (for an order 3 symmetry) or 1 (for an order 2 symmetry).
\end{lm}

\begin{prop}[\cite{deg}]
Assume that $j(z_1, z_2, z_3)\in\R$. Then, the following holds
\begin{itemize}
\item
The $j$-invariant $j(z_1, z_2, z_3)<1$ if and only if the points $z_1, z_2, z_3$ form an isosceles triangle. The special angle seen as a function of the $j$-invariant is a increasing monotone function. This angle tends to $0$ when $j$ tends to $-\infty$, equals $\frac{\pi}{3}$ at $j=0$ and tends to $\frac{\pi}{2}$ when $j$ approaches $1$. 

\item
The $j$-invariant $j(z_1, z_2, z_3)\geq1$ if and only if the points $z_1, z_2, z_3$ are collinear.The ratio between the lengths of the smallest segment and the longest segment $\overline{z_lz_k}$ seen as a function of the $j$-invariant is a decreasing monotone function. This ratio equals $1$ when $j$ equals $1$, and $0$ when $j$ approaches $\infty$.
\end{itemize}
\end{prop}

As a corollary, if the $j$-invariant of $\{z_1, z_2, z_3\}$ is not real, then the points form a triangle having side lengths pairwise different. Therefore, in this case, the points $z_1, z_2, z_3$ can be ordered according to the increasing order of side lengths of the opposite edges. E.g., for $\{z_1, z_2, z_3\}=\{\frac{i}{2},0,1\}$ the order is $1$, $\frac{i}{2}$ and $0$.

\begin{prop}[\cite{deg}]
 If $j(z_1, z_2, z_3)\notin\R$, then the above order on the points $z_1, z_2, z_3$ is clockwise if $\Im(j(z_1, z_2, z_3))>0$ and anti-clockwise if $\Im(j(z_1, z_2, z_3))<0$.
\end{prop}

\subsubsection{The $j$-invariant of a trigonal curve}
Let $C$ be a proper trigonal curve. We use the $j$-invariant defined for triples of complex numbers in order to define a meromorphic map $j_C$ on the base curve $B$. The map $j_{C}$ encodes the topology of the trigonal curve $C$. The map $j_C$ is called the $j$-\emph{invariant} of the curve $C$ and provides a correspondence between trigonal curves and dessins.

\begin{df}
 For a proper trigonal curve~$C$, we define its $j$-invariant $j_C$ as the analytic continuation of the map
 \[
\begin{array}{ccc}
 B^{\#}&\lra&\C \\
 b &\lmt &j\mbox{-invariant of } C\cap F_b^0\subset F_b^0\cong\C.
\end{array}
\]
We call the trigonal curve $C$ {\it isotrivial} if its $j$-invariant is constant.
\end{df}

If a proper trigonal curve $C$ is given by a Weierstraß equation of the form \eqref{eq:Weiertrass}, then
\begin{equation}
j_C=-\frac{4g_2^3}{\Delta}\; , \text{ where }\Delta=-4g_2^3-27g_3^2.
\end{equation}

\begin{thm}[\cite{deg}] \label{thm:Etc}
 Let $B$ be a compact curve and $j\colon B\lra\CP$ a non-constant meromorphic map. Up to Nagata equivalence, there exists a unique trigonal curve $C\subset\Si\lra B$ such that $j_{C}=j$.
\end{thm}

Following the proof of the theorem, $j_B\lra\CP$ leads to a unique \emph{minimal} proper trigonal curve $C_j$, in the sense that any other trigonal curve with the same $j$-invariant can be obtained by positive Nagata transformations from $C_j$.

An equisingular deformation $C_s$, $s\in S$, of $C$ leads to an analytic deformation of the couple $(B_s,j_{C_s})$.

\begin{coro}[\cite{deg}] \label{coro:Etc}
 Let $(B,j)$ be a couple, where $B$ is a compact curve and $j\colon B\lra\CP$ is a non-constant meromorphic map. Then, any deformation of $(B,j)$ results in a deformation of the minimal curve $C_j\subset\Si\lra B$ associated to $j$.
\end{coro}

The $j$-invariant of a generic trigonal curve $C\subset\Sigma\lra B$ has degree $\deg(j_C)=6d$, where $d=-E^2$. A positive Nagata transformation increases $d$ by one while leaving $j_C$ invariant.
The $j$-invariant of a generic trigonal curve $C$ has a ramification index equal to $3$, $2$ or $1$ at every point $b\in B$ such that $j_C(b)$ equals $0$, $1$ or $\infty$, respectively. We can assume, up to perturbation, that every critical value of $j_C$ is simple. In this case we say that $j_C$ has a {\it generic branching behavior}.

\subsubsection{Real structures}\label{sssec:realstruc}
We are mostly interested in real trigonal curves. A {\it real structure} on a complex variety $X$ is an anti-holomorphic involution $c\colon X\lra X$. We define a {\it real variety} as a couple $(X,c)$, where $c$ is a real structure on a complex variety $X$. We denote by $X_{\R}$ the fixed point set of the involution $c$ and we call $X_{\R}$ {\it the set of real points} of $c$.

We say that a real curve $(X,c)$ is of \emph{type~$\mathrm{I}$} if $\widetilde{X}\setminus \widetilde{X}_{\R}$ is disconnected, where $\widetilde{X}$ is the normalization of $X$.

\subsubsection{Real trigonal curves}
\label{sssec:rtc}

A geometrically ruled surface $\pi\colon\Si\lra B$ is \emph{real} if there exist real structures $c_{\Si}\colon\Si\lra\Si$ and $c_{B}\colon B\lra B$ compatible with the projection~$\pi$, i.e., such that $\pi\circ c_{\Si}=c_{B}\circ\pi$. We assume the exceptional section is {\it real} in the sense that it is invariant by conjugation, i.e., $c_{\Si}(E)=E$. 
Put $\pi_{\R}:=\pi|_{\Si_{\R}}\colon{\Si_{\R}\lra B_{\R}}$. Since the exceptional section is real, the fixed point set of every fiber is not empty, implying that the real structure on the fiber is isomorphic to the standard complex conjugation on $\CP$. Hence all the fibers of $\pi_{\R}$ are isomorphic to~$\RP$.
Thus, the map $\pi_{\R}$ establishes a bijection between the connected components of the real part $\Si_{\R}$ of the surface $\Si$ and the connected components of the real part~$B_{\R}$ of the curve $B$. Every connected component of $\Si_{\R}$ is homeomorphic either to a torus or to a Klein bottle.

The ruled surface $\Sigma$ can be seen as the fiberwise projectivization of a rank 2 vector bundle over~$B$.
Let us assume $\Sigma=\mathbf{P}(1\oplus \mathcal{Y})$, where $1$ is the trivial line bundle over~$B$ and $\mathcal{Y}\in\operatorname{Pic}(B)$. We put $\mathcal{Y}_{i}:=\mathcal{Y}_{\R}|_{B_{i}}$ for every connected component $B_{i}$ of~$B_{\R}$. Hence $\Si_{i}:=\Si_{\R}|_{B_{i}}$ is orientable if and only if $\mathcal{Y}_{i}$ is topologically trivial, i.e., its first Stiefel-Whitney class $w_{1}(\mathcal{Y}_{i})$ is zero.

\begin{df}
 A \emph{real trigonal curve} $C$ is a trigonal curve contained in a real ruled surface $(\Si,c_{\Si})\lra (B,c_{B})$ such that $C$ is $c_{\Si}$-invariant, i.e., $c_{\Si}(C)=C$.
\end{df}

The line bundle $\mathcal{Y}$ can inherit a real structure from $(\Sigma,c_{\Sigma})$. 
If a real trigonal curve is proper, then its $j$-invariant is real, seen as a morphism $j_{C}\colon(B,c_{B})\lra (\CP,z\lmt\bar{z})$, where $z\lmt\bar{z}$ denotes the standard complex conjugation on $\CP$. In addition, the sections $g_{2}$ and $g_{3}$ can be chosen real.

 Let us consider the restriction $\pi|_{C_{\R}}\colon C_{\R}\lra B_{\R}$. We put $C_{i}:=\pi|_{C_{\R}}^{-1}(B_{i})$ for every connected component $B_{i}$ of $B_{\R}$. We say that $B_{i}$ is \emph{hyperbolic} if $\pi|_{C_{i}}\colon C_{i}\lra~B_{i}$ has generically a fiber with three elements. The trigonal curve $C$ is \emph{hyperbolic} if its real part is non-empty and all the connected components of $B_{\R}$ are hyperbolic.

\begin{df} \label{df:ovzz}
 Let $C$ be a non-singular generic real trigonal curve.
A connected component of the set $\{b\in B\mid \operatorname{Card}(\pi|_{C_{\R}}^{-1}(b))\geq2\}$ is an \emph{oval} if it is not a hyperbolic component and its preimage by $\pi|_{C_{\R}}$ is disconnected. Otherwise, the connected component is called a \emph{zigzag}.
\end{df}

\subsection{Dessins}

The dessins d'enfants were introduced by A. Grothendieck (cf.~\cite{schneps})
in order to study the action of the absolute Galois group. We use a modified version of dessins d'enfants which was proposed by S. Orevkov~\cite{Orevkov}.

\subsubsection{Trichotomic graphs}

Let $S$ be a compact connected topological surface. A graph $D$ on the surface $S$ is a graph embedded into the surface and considered as a subset $D\subset S$. We denote by $\operatorname{Cut}(D)$ the \emph{cut} of $S$ along $D$, i.e., the disjoint union of the closure of connected components of $S\setminus D$.

\begin{df} \label{df:trigra}
 A \emph{trichotomic graph} on a compact surface $S$ is an embedded finite directed graph $D\subset S$ decorated with the following additional structures (referred to as \emph{colorings} of the edges and vertices of $D$, respectively):
 
\begin{itemize}
\item every edge of $D$ is color solid, bold or dotted,
\item every vertex of $D$ is black ($\bullet$), white ($\circ$), cross ($\times$) 
or monochrome (the vertices of the first three types are called \emph{essential}),
\end{itemize}
and satisfying the following conditions:
\begin{enumerate}
\item $\partial S\subset D$,
\item every essential vertex is incident to at least $2$ edges,
\item every monochrome vertex is incident to at least $3$ edges,
\item the orientations of the edges of $D$ form an orientation of the boundary $\partial\operatorname{Cut} (D)$ which is compatible with an orientation on $\operatorname{Cut} (D)$,
\item all edges incident to a monochrome vertex are of the same color,
\item $\times$-vertices are incident to incoming dotted edges and outgoing solid edges,
\item $\bullet$-vertices are incident to incoming solid edges and outgoing bold edges,
\item $\circ$-vertices are incident to incoming bold edges and outgoing dotted edges.
\end{enumerate}
\end{df}

Let $D\subset S$ be a trichotomic graph. A \emph{region} $R$ is an element of $\operatorname{Cut(D)}$. The boundary $\partial R$ of $R$ contains $n=3k$ essential vertices. 
A region with $n$ vertices on its boundary is called an \emph{$n$-gonal region}.
We denote by $D_{solid}$, $D_{bold}$, $D_{dotted}$ the monochrome parts of $D$, i.e., the sets of vertices and edges of the specific color. On the set of vertices of a specific color, we define the relation $u\preceq v$ if there is a monochrome path from $u$ to~$v$, i.e., a path formed entirely of edges and vertices of the same color. We call the graph $D$ \emph{admissible} if the relation $\preceq$ is a partial order, equivalently, if there are no directed monochrome cycles.

\begin{df}
 A trichotomic graph $D$ is a \emph{dessin} if
\begin{enumerate}
 \item $D$ is admissible;
 \item every trigonal region of $D$ is homeomorphic to a disk.
\end{enumerate}
\end{df}

The orientation of the graph $D$ is determined by the pattern of colors of the vertices on the boundary of every region.

\subsubsection{Complex and real dessins}

Let $S$ be an orientable surface. Every orientation of $S$ induces a \emph{chessboard coloring} of $\Cut(D)$, i.e., a function on $\Cut(D)$ determining if a region $R$ endowed with the orientation set by $D$ coincides with the orientation of $S$.

\begin{df} A \emph{real trichotomic graph} on a real closed surface $(S,c)$ is a trichotomic graph $D$ on $S$ which is invariant under the action of $c$. Explicitly, every vertex $v$ of $D$ has as image $c(v)$ a vertex of the same color; every edge $e$ of $D$ has as image $c(e)$ an edge of the same color.
\end{df}

Let $D$ be a real trichotomic graph on $(S,c)$. Let $\overline{S}:=S/c$ be the quotient surface and put $\overline{D}\subset\overline{S}$ as the image of $D$ by the quotient map $S\lra S/c$. The graph $\overline{D}$ is a well defined trichotomic graph on the surface $S/c$.

In the inverse sense, let $S$ be a compact surface, which can be non-orientable or can have non-empty boundary. Let $D\subset S$ be a trichotomic graph on $S$.
Consider its complex double covering $\widetilde{S}\lra S$ (cf. \cite{AG} for details), which has a real structure given by the deck transformation, and put $\widetilde{D}\subset\widetilde{S}$ the inverse image of $D$. The graph~$\widetilde{D}$ is a graph on $\widetilde{S}$ invariant by the deck transformation. We use these constructions in order to identify real trichotomic graphs on real surfaces with their images on the quotient surface.

\begin{lm}[\cite{deg}]
 Let $D$ be a real trichotomic graph on a real closed surface~$(\widetilde{S},c)$. Then, every region $R$ of $D$ is disjoint from its image $c(R)$.
\end{lm}

\begin{prop}[\cite{deg}] \label{prop:realdd}
Let $S$ be a compact surface. Given a trichotomic graph $D\subset S$, then its oriented double covering $\widetilde{D}\subset\widetilde{S}$ is a real trichotomic graph. Moreover, $\widetilde{D}\subset\widetilde{S}$ is a dessin if and only if so is $D\subset S$.
Conversely, if $(S,c)$ is a real compact surface and $D\subset S$ is a real trichotomic graph, then its image $\overline{D}$ in the quotient $\overline{S}:=S/c$ is a trichotomic graph. Moreover, $\overline{D}\subset\overline{S}$ is a dessin if and only if so is $D\subset S$. 
\end{prop}

\begin{df}
Let $D$ be a dessin on a compact surface $S$. Let us denote by $\Ver(D)$ the set of vertices of $D$. For a vertex $v\in\Ver(D)$, we define the \emph{index} $\Ind(v)$ of $v$ as half of the number of incident edges of $\widetilde{v}$, where $\widetilde{v}$ is a preimage of $v$ by the double complex cover of $S$ as in Proposition~\ref{prop:realdd}.

A vertex $v\in\Ver(D)$ is \emph{singular} if
\begin{itemize}
\item $v$ is black and $\Ind(v)\not\equiv0\mod3$,
\item or $v$ is white and $\Ind(v)\not\equiv0\mod2$,
\item or $v$ has color $\times$ and $\Ind(v)\geq2$.
\end{itemize}
We denote by $\Sing(D)$ the set of singular vertices of $D$. A dessin is \emph{non-singular} if none of its vertices is singular.
\end{df}

\begin{df}
 Let $B$ be a complex curve and let $j\colon B\lra\CP$ a non-constant meromorphic function, in other words, a ramified covering of the complex projective line. The dessin $D:=\Dssn(j)$ associated to $j$ is the graph given by the following construction:
\begin{itemize}
\item as a set, the dessin $D$ coincides with $j^{-1}(\RP)$, where $\RP$ is the fixed point set of the standard complex conjugation in $\CP$;
\item black vertices $(\bullet)$ are the inverse images of $0$;
\item white vertices $(\circ)$ are the inverse images of $1$;
\item vertices of color $\times$ are the inverse images of $\infty$;
\item monochrome vertices are the critical points of $j$ with critical value in $\CP\setminus\{0,1,\infty\}$;
\item solid edges are the inverse images of the interval $[\infty,0]$;
\item bold edges are the inverse images of the interval $[0,1]$;
\item dotted edges are the inverse images of the interval $[1,\infty]$;
\item orientation on edges is induced from an orientation of $\RP$.
\end{itemize}
\end{df}

\begin{lm}[\cite{deg}]
Let $S$ be an oriented connected closed surface. Let $j\colon S\lra\CP$ a ramified covering map. The trichotomic graph $D=\Dssn(j)\subset S$ is a dessin. Moreover, if $j$ is real with respect to an orientation-reversing involution $c\colon S\lra S$, then $D$ is $c$-invariant.
\end{lm}

Let $(S,c)$ be a compact real surface. If $j\colon (S,c)\lra (\CP,z\lra\bar{z})$ is a real map, we define $\Dssn_{c}(j):=\Dssn(j)/c\subset S/c$.

\begin{thm}[\cite{deg}]
Let $S$ be an oriented connected closed surface (and let $c\colon S\lra S$ a orientation-reversing involution). A (real) trichotomic graph $D\subset S$ is a (real) dessin if and only if $D=\Dssn(j)$ 
for a (real) ramified covering $j\colon S\lra\CP$. 

Moreover, $j$ is unique up to homotopy in the class of (real) ramified coverings with dessin $D$.
\end{thm}

The last theorem together with the Riemann existence theorem provides the next corollaries, for the complex and real settings.

\begin{coro}[\cite{deg}] \label{coro:Ej}
Let $D\subset S$ be a dessin on a compact closed orientable surface $S$. Then there exists a complex structure on $S$ and a holomorphic map $j\colon S\lra\CP$ such that $\Dssn(j)=D$. Moreover, this structure is unique up to deformation of the complex structure on $S$ and the map $j$ in the Kodaira-Spencer sense.
\end{coro}

\begin{coro}[\cite{deg}] \label{coro:EjR}
Let $D\subset S$ be a dessin on a compact surface $S$. Then there exists a complex structure on its double cover $\widetilde{S}$ and a holomorphic map $j\colon\widetilde{S}\lra\CP$ such that $j$ is real with respect to the real structure $c$ of $\widetilde{S}$ and $\Dssn_{c}(j)=D$. Moreover, this structure is unique up to equivariant deformation of the complex structure on $S$ and the map $j$ in the Kodaira-Spencer sense.
\end{coro}

\subsubsection{Deformations of dessins}

In this section we describe the notions of deformations which allow us to associate classes of non-isotrivial trigonal curves and classes of dessins, up to deformations and equivalences that we explicit.

\begin{df}
A {\it deformation of coverings} is a homotopy $S\times [0,1] \lra \CP$ within the class of (equivariant) ramified coverings. The deformation is {\it simple} if it preserves the multiplicity of the inverse images of $0$, $1$, $\infty$ and of the other real critical values.
\end{df}

Any deformation is locally simple except for a finite number of values~$t\in [0,1]$.

\begin{prop}[\cite{deg}]
 Let $j_{0}, j_{1}\colon S\lra\CP$ be ($c$-equivariant) ramified coverings. They can be connected by a simple (equivariant) deformation if and only their dessins $D(j_0)$ and $D(j_{1})$ are isotopic (respectively, $D_{c}(j_{0})$ and $D_{c}(j_{1})$).
\end{prop}

\begin{df} \label{df:equidef}
 A deformation $j_{t}\colon S\lra\CP$ of ramified coverings is \emph{equisingular} if the union of the supports 
\[\bigcup_{t\in [0,1]} \operatorname{supp}\left\{ (j^{*}_{t}(0)\mod 3)+(j^{*}_{t}(1)\mod 2)+j^{*}_{t}(\infty)\right\}\] 
considered as a subset of $S\times [0,1]$ is an isotopy. Here $^{*}$ denotes the divisorial pullback of a map $\varphi:S\lra S'$ at a point $s'\in S'$:
 \[ \varphi^{*}(s')=\displaystyle\sum_{s\in\varphi^{-1}(s')}r_{s}s, \] where $r_{s}$ if the ramification index of $\varphi$ at $s\in S$.
\end{df}

A dessin $D_1\subset S$ is called a \emph{perturbation} of a dessin $D_0\subset S$, and $D_0$ is called a \emph{degeneration} of $D_1$, if for every vertex $v\in \Ver(D_0)$ there exists a small neighboring disk $U_v\subset S$ such that $D_0\cap U_v$ only has edges incident to $v$,  $D_1\cap U_v$ contains essential vertices of at most one color, and $D_0$ and $D_1$ coincide outside of $U_v$.

\begin{thm}[\cite{deg}] \label{thm:support}
 Let $D_{0}\subset S$ be a dessin, and let $D_{1}$ be a perturbation. Then there exists a map $j_{t}\colon S\lra\CP$ such that
\begin{enumerate}
 \item $D_{0}=\Dssn(j_0)$ and $D_{1}=\Dssn(j_{1})$;
 \item $j_{t}|_{S\setminus\bigcup_{v}U_{v}}=j_{t'}|_{S\setminus\bigcup_{v}U_{v}}$ for every $t$, $t'$;
 \item the deformation restricted to $S\times (0,1]$ is simple.
\end{enumerate}
\end{thm}

\begin{coro}[\cite{deg}] \label{coro:def}
 Let $S$ be a complex compact curve, $j\colon S\lra\CP$ a non-constant holomorphic map, and let $\Dssn(j)=D_{0}, D_{1}, \dots, D_{n}$ be a chain of dessins in $S$ such that for $i=1, \dots, n$ either $D_{i}$ is a perturbation of $D_{i-1}$, or $D_{i}$ is a degeneration of $D_{i-1}$, or $D_{i}$ is isotopic to $D_{i-1}$. Then there exists a piecewise-analytic deformation $j_{t}\colon S_{t}\lra\CP$, $t\in[0,1]$, of $j_0=j$ such that $\Dssn(j_{1})=D_{n}$.
\end{coro}

\begin{coro}[\cite{deg}] \label{coro:defR}
 Let $(S,c)$ be a real compact curve, $j\colon (S,c)\lra(\CP,z\lmt\bar{z})$ be a real non-constant holomorphic map, and let $\Dssn_{c}(j)= D_{0}, D_{1}, \dots, D_{n}$ be a chain of real dessins in $(S,c)$ such that for $i=1, \dots, n$ either $D_{i}$ is a equivariant perturbation of $D_{i-1}$, or $D_{i}$ is a equivariant degeneration of $D_{i-1}$, or $D_{i}$ is equivariantly isotopic to $D_{i-1}$. Then there is a piecewise-analytic real deformation $j_{t}\colon (S_{t},c_{t})\lra(\CP,\bar{\cdot})$, $t\in[0,1]$, of $j_0=j$ such that $\Dssn_{c}(j_{1})=D_{n}$.
\end{coro}

Due to Theorem \ref{thm:support}, the deformation $j_t$ given by Corollaries \ref{coro:def} and \ref{coro:defR} is equisingular in the sense of Definition \ref{df:equidef} if and only if all perturbations and degenerations of the dessins on the chain $D_{0}, D_{1}, \dots, D_{n}$ are equisingular.

\subsection{Trigonal curves and their dessins}

In this section we describe an equivalence between dessins.

\subsubsection{Correspondence theorems}

Let $C\subset\Si\lra B$ be a non-isotrivial proper trigonal curve.
We associate to $C$ the dessin corresponding to its $j$-invariant $\Dssn(C):=\Dssn(j_{C})\subset B$. 
In the case when $C$ is a real trigonal curve we associate to $C$ the image of the real dessin corresponding to its $j$-invariant under the quotient map, $\Dssn_{c}(C):=\Dssn(j_{C})\subset B/{c_{B}}$, where~$c_{B}$ is the real structure of the base curve $B$.\\

So far, we have focused on one direction of the correspondences: we start with a trigonal curve $C$, 
consider its $j$-invariant and 
construct the dessin associated to it.  
Now, we study the opposite direction. 
Let us consider a dessin $D$ on a topological orientable closed surface $S$. By Corollary \ref{coro:Ej},
there exist a complex structure~$B$ on~$S$ and a holomorphic map $j_{D}\colon B\lra\CP$ such that $\Dssn(j_{D})=D$. By Theorem~\ref{thm:Etc} and Corollary~\ref{coro:Etc} there exists a trigonal curve~$C$ having~$j_{D}$ as $j$-invariant; such a curve is unique up to deformation in the class of trigonal curves with fixed dessin. Moreover, due to Corollary~\ref{coro:def}, any sequence of isotopies, perturbations and degenerations of dessins gives rise to a piecewise-analytic deformation of trigonal curves, which is equisingular if and only if all perturbations and degenerations are.

In the real framework, let $(S,c)$ a compact close oriented topological surface endowed with a orientation-reversing involution. Let $D$ be a real dessin on $(S,c)$. By Corollary~\ref{coro:EjR}, there exists a real structure $(B,c_{B})$ on $(S,c)$ and a real map \linebreak[1]$j_{D}\colon(B,c_{B})\lra(\CP,z\longmapsto\bar{z})$ such that $\Dssn_c(j_{D})=D$. By Theorem \ref{thm:Etc}, Corollary \ref{coro:Etc} and the remarks made in Section \ref{sssec:rtc}, there exists a real trigonal curve~$C$ having~$j_{D}$ as $j$-invariant; such a curve is unique up to equivariant deformation in the class of real trigonal curves with fixed dessin. Furthermore, due to Corollary~\ref{coro:defR}, any sequence of isotopies, perturbations and degenerations of dessins gives rise to a piecewise-analytic equivariant deformation of real trigonal curves, which is equisingular if and only if all perturbations and degenerations are.

\begin{df}
A dessin is \emph{reduced} if
\begin{itemize}
\item 
for every $v$ $\bullet$-vertex one has $\Ind{v}\leq3$,
\item 
for every $v$ $\circ$-vertex one has $\Ind{v}\leq2$,
\item every monochrome vertex is real and has index $2$.
\end{itemize}
A reduced dessin is {\it generic} if all its $\bullet$-vertices and $\circ$-vertices are non-singular and all its $\times$-vertices have index~$1$.
\end{df}

Any dessin admits an equisingular perturbation to a reduced dessin. The vertices with excessive index (i.e., index greater than 3 for $\bullet$-vertices or than 2 for $\circ$-vertices) can be reduced by introducing new vertices of the same color.

In order to define an equivalence relation of dessins, we introduce 
{\it elementary moves}. Consider two reduced dessins $D$, $D'\subset S$ such that they coincide outside a closed disk $V\subset S$.
If $V$ does not intersect $\partial S$ and the graphs  $D\cap V$ and $D'\cap V$ are as shown in Figure~\ref{fig:elem}(a), then we say that performing a {\it monochrome modification} on the edges intersecting $V$ produces $D'$ from $D$, or {\it vice versa}.
This is the first type of {\it elementary moves}. 
Otherwise, the boundary component inside $V$ is shown in light gray. In this setting, if the graphs $D\cap V$ and $D'\cap V$ are as shown in one of the subfigures in Figure~\ref{fig:elem}, we say that performing an \emph{elementary move} of the corresponding type on $D\cap V$ produces $D'$ from $D$, or {\it vice versa}. 

\begin{df} 
Two reduced dessins $D$, $D'\subset S$ are {\it elementary equivalent} if, after a (preserving orientation, in the complex case) homeomorphism of the underlying surface $S$ they can be connected by a sequence of isotopies and elementary moves between dessins, as described in Figure~\ref{fig:elem}.
\end{df}

\begin{figure} 
\centering
\includegraphics[width=5in]{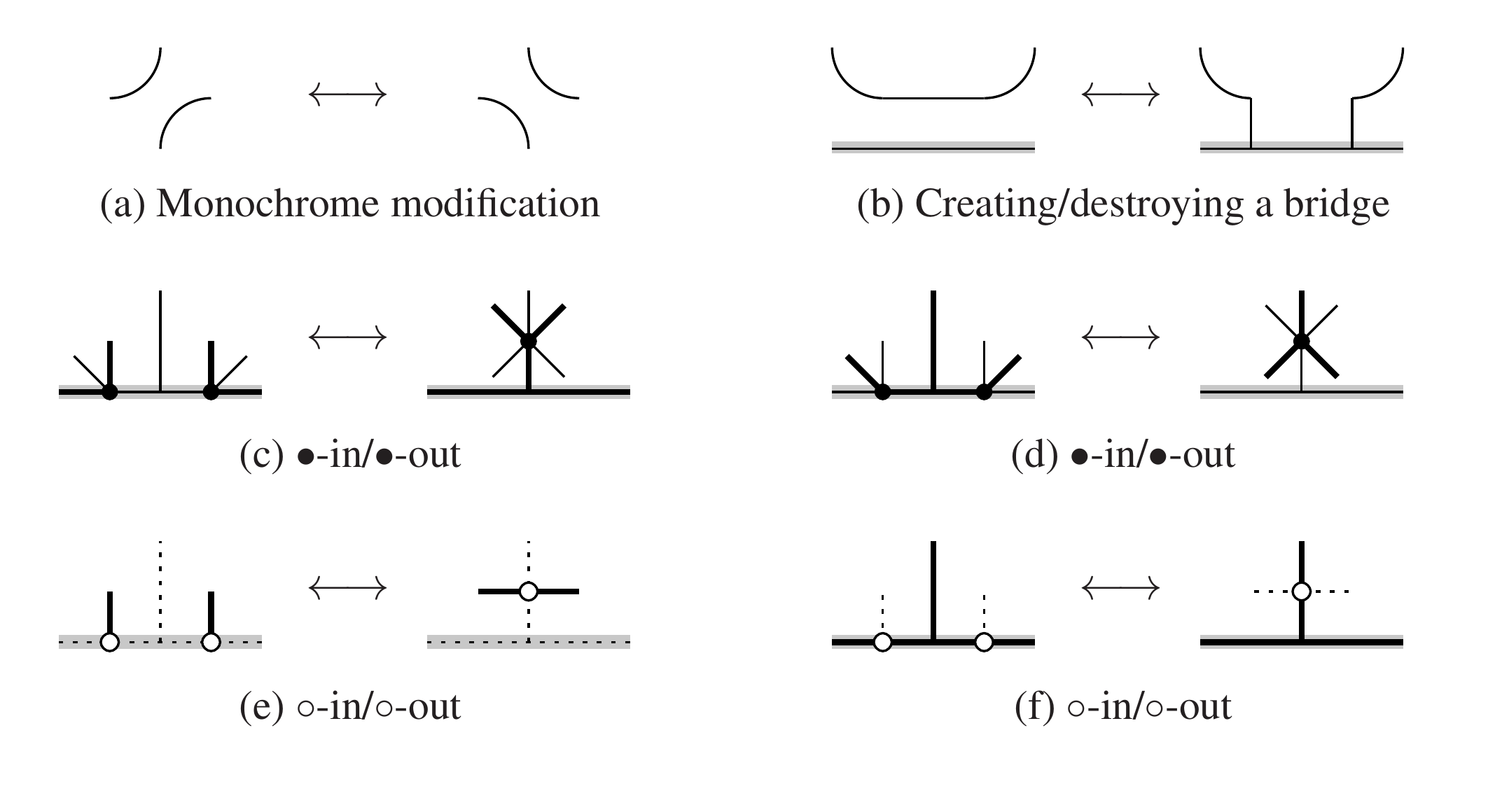}
\caption{Elementary moves.\label{fig:elem}}
\end{figure}

This definition is meant so that two reduced dessins are elementary equivalent if and only if they can be 
connected up to homeomorphism by a sequence of isotopies, equisingular perturbations and degenerations.

The following theorems establish the equivalences between the deformation classes of trigonal curves we are interested in and elementary equivalence classes of certain dessins. We use these links to obtain different classifications of curves {\it via} the combinatorial study of dessins.

\begin{thm}[\cite{deg}] \label{th:correpondance1}
 There is a one-to-one correspondence between the set of equisingular deformation classes of non-isotrivial proper trigonal curves ${C\subset\Si\lra B}$ with $\widetilde{A}$ type singular fibers only and the set of elementary equivalence classes of reduced dessins $D \subset B$.
\end{thm}

 \begin{thm}[\cite{deg}] \label{th:correpondance2}
 There is a one-to-one correspondence between the set of equivariant equisingular deformation classes of non-isotrivial proper real trigonal\linebreak[1] curves $C\subset\Si\lra(B,c)$ with $\widetilde{A}$ type singular fibers only and the set of elementary equivalence classes of reduced real dessins $D \subset B/c$.
\end{thm}

\begin{thm}[\cite{deg}] \label{th:correpondance3}
 There is a one-to-one correspondence between the set of equivariant equisingular deformation classes of almost generic real trigonal\linebreak[1] curves $C \subset\Si\lra(B,c)$ and the set of elementary equivalence classes of generic real dessins $D \subset B/c$.
\end{thm}

This correspondences can be extended to trigonal curves with more general singular fibers (see \cite{deg}).

\begin{df} 
Let $C\subset\Si\lra B$ be a proper trigonal curve. 
We define the {\it degree} of the curve $C$ as $\deg(C):=-3E^2$ where $E$ is the exceptional section of $\Si$. For a dessin $D$, we define its {\it degree} as $\deg(D)=\deg(C)$ where $C$ is a minimal proper trigonal curve such that $\Dssn(C)=D$.
\end{df}

\subsubsection{Real generic curves}

Let $C$ be a generic real trigonal curve and let $D:=\Dssn_{c}(C)$ be a generic dessin. 
The {\it real part} of $D\subset S$ is the intersection $D\cap\partial S$. For a specific color $\ast\in\{\mathrm{ solid},\mathrm{ bold},\mathrm{ dotted}\}$, $D_{\ast}$ is the subgraph of the corresponding color and its adjacent vertices. The components of $D_{\ast}\cap\partial S$ are either components of~$\partial S$, called \emph{monochrome components} of $D$, or 
segments, called {\it maximal monochrome segments} of $D$. 
We call these monochrome components or segments \emph{even} or \emph{odd} according to 
the parity of the number of $\circ$-vertices they contain.

Furthermore, we refer to the dotted monochrome components as \emph{hyperbolic components}. A dotted segment without $\times$-vertices of even index is referred to as an \emph{oval} if it is even, or as a \emph{zigzag} if it is odd.

\begin{df} 
Let $D\subset S$ be a real dessin. Assume that there is a subset of $S$ in which $D$ has a configuration of vertices and edges as in Figure~\ref{fig:zigzag}. Replacing the corresponding configuration with the alternative one defines another dessin $D'\subset S$. We say that $D'$ is obtained from $D$ by \emph{straightening/creating} a zigzag.

Two dessins $D,D'$ are \emph{weakly equivalent} if there exists a finite sequence of dessins $D=D_0, D_1, \dots,D_n=D'$ such that $D_{i+1}$ is either elementary equivalent to $D_{i}$, or $D_{i+1}$ is obtained from $D_i$ by straightening/creating a zigzag.
\end{df}

Notice that if $D'$ is obtained from $D$ by straightening/creating a zigzag and $\widetilde{D},\widetilde{D'}\subset\widetilde{S}$ are the liftings of $D$ and $D'$ in $\widetilde{S}$, the double complex of $S$, then $\widetilde{D}$ and $\widetilde{D'}$ are elementary equivalent as complex dessins. However, $D$ and $D'$ are not elementary equivalent, since the number of zigzags of a real dessin is an invariant on the elementary equivalence class of real dessins.

\begin{figure}
\centering 
\includegraphics[width=5in]{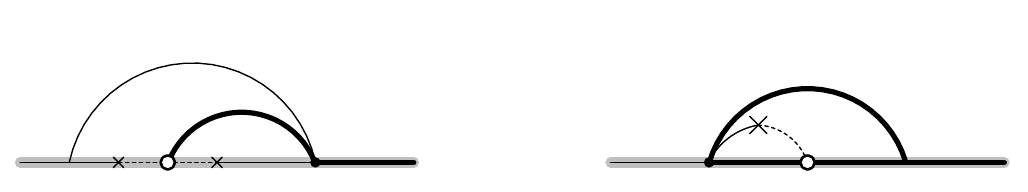}
\caption{Straightening/creating a zigzag.\label{fig:zigzag}}
\end{figure}

\subsubsection{Type of a dessin}
 Let $C\subset\Si\lra B$ be a real trigonal curve over a base curve of type~$I$. We define $C_{\mathrm{Im}}$ as the closure of the set $\pi^{-1}|_{C}(B_{\R})\setminus C_{\R}$ and let $B_{\mathrm{\Im}}=\pi(C_{\mathrm{\Im}})$. By definition $C_{\mathrm{Im}}$ is invariant with respect to the real structure of~$C$. Moreover, $C_{\mathrm{Im}}=\emptyset$ if and only if $C$ is a hyperbolic trigonal curve.

\begin{lm}[\cite{DIZ}] 
 A trigonal curve is of type~$\mathrm{I}$ if and only if the homology class $[C_{\mathrm{Im}}]\in H_{1}(C;\Z/2\Z)$ is zero.
\end{lm}

In view of the last lemma, we can represent a trigonal curve of type~$\mathrm{I}$ as the union of two orientable surfaces $C_{+}$ and $C_{-}$, intersecting at their boundaries $C_{+}\cap C_{-}=\partial C_{+}=\partial C_{-}$. Both surfaces, $C_{+}$ and $C_{-}$, are invariant under the real structure of~$C$. We define 
\[
m_{\pm}:
\begin{array}{ccc}
		B &	\lra	& \Z \\
		b & 	\lmt	&	\operatorname{Card}(\pi|_{C}^{-1}(b)\cap C_{\pm})-\chi_{B_{\mathrm{Im}}}(b), 
\end{array}
\]
where $\chi_{B_{\mathrm{Im}}}$ is the characteristic function of the set $B_{\mathrm{Im}}$. These maps are locally constant over $B^{\#}$, and since $B^{\#}$ is connected, the maps are actually constant. Moreover, $m_{+}+m_{-}=3$, so we choose the surfaces $C_{\pm}$ such that $m_{+}|_{B^{\#}}\equiv 1$ and $m_{-}|_{B^{\#}}\equiv 2$. 

We can label each region $R\in\operatorname{Cut}(D)$ where $D=\Dssn(C)$ according to the label on $C_{+}$. Given any point $b\in R$ on the interior, the vertices of the triangle $\pi|_{C}^{-1}(b)\subset F_b$ are labeled by 1, 2, 3, according to the
increasing order of lengths of the opposite side of the triangle. We label the region~$R$ by the label of the point $\pi|_{C}^{-1}(b)\cap C_{+}$. 

We can also label the interior edges according to the adjacent regions in the following way:
\begin{itemize}
 \item every solid edge can be of type~$1$ (i.e., both adjacent regions are of type 1) or type~$\overline{1}$ (i.e., one region of type~$2$ and one of type~$3$);
 \item every bold edge can be of type~$3$ or type~$\overline{3}$;
 \item every dotted edge can be of type~$1$, $2$ or $3$.
\end{itemize}
We use the same rule for the real edges of $D$. Note that there are no real solid edges of type~$1$ nor real bold edges of type~$3$ (otherwise the morphism $C_{+}\lra B$ would have two layers over the regions of $D$ adjacent to the edge).

\begin{thm}[\cite{DIZ}]\label{df:type}
 A generic non-hyperbolic curve $C$ is of type~$\mathrm{I}$ if and only if the regions of $D$ admit a labeling which satisfies the conditions described above.
\end{thm}

\section{Uninodal dessins}\label{ch:uni}

Generic trigonal curves are smooth.
Smooth proper trigonal curves have non-singular dessins.
Singular proper trigonal curves have singular dessins and the singular points are represented by singular vertices. 
A generic singular trigonal curve $C$ has exactly one singular point, which is a non-degenerate double point ({\it node}). 
Moreover, if $C$ is a proper trigonal curve, then the double point on it is represented by a $\times$-vertex of index $2$ on its dessin.
In addition, if $C$ has a real structure,
the double point is real and so is its corresponding vertex, leading to the cases where the $\times$-vertex of index $2$ has dotted real edges (representing the intersection of two real branches) or has solid real edges (representing one isolated real point, which is the intersection of two complex conjugated branches).

\begin{df} 
Let $D\subset S$ be a dessin on a compact surface $S$. A {\it nodal vertex} 
({\it node})  
of $D$ is a $\times$-vertex of index $2$.
The dessin $D$ is called \emph{nodal} if all its singular vertices are nodal vertices; it is called \emph{uninodal} if it has exactly one singular vertex which is a node.
We call a \emph{toile} a non-hyperbolic real nodal dessin on $(\CP,z\lmt\bar{z})$. 
\end{df}

Since a real dessin on $(\CP,z\lmt\bar{z})$ descends to the quotient, we represent toiles on the disk.

In a real dessin, there are two types of real nodal vertices, namely, vertices having either real solid edges and interior dotted edges, or dotted real edges and interior solid edges. We call \emph{isolated nodes} of a dessin $D$ thoses $\times$-vertices of index $2$ corresponding to the former case and \emph{non-isolated nodes} those corresponding to the latter. 

\begin{df} 
Let $D\subset S$ be a real dessin.
A {\it bridge} of $D$ is an edge $e$ contained in a connected component of the boundary $\partial S$ having more than two vertices, such that $e$ connects two monochrome vertices. The dessin $D$ is called {\it bridge-free} if it has no bridges.
The dessin $D$ is called {\it peripheral} if it has no inner vertices other than $\times$-vertices.
\end{df}

For non-singular dessins, combinatorial statements analogous to Lemma~\ref{lm:bf} and Proposition~\ref{prop:peripheral} are proved in \cite{DIK}.

\begin{lm} \label{lm:bf} \label{cr:bf}
  A nodal dessin $D$ is elementary equivalent to a bridge-free dessin $D'$ having the same number of inner essential vertices and real essential vertices.
\end{lm}

\begin{proof}
Let $D$ be a dessin on $S$. Let $e$ be a bridge of $D$ lying on a connected component of $\partial S$. Let $u$ and $v$ be the endpoints of $e$. Since $e$ is a bridge, there exists at least one real vertex $w\neq u$ of $D$ adjacent to $v$. If $w$ is an essential vertex, destroying the bridge is an admissible elementary move. 
Otherwise, the vertex $w$ is monochrome of the same type as $u$ and $v$, and the edge connecting $v$ and $w$ is another bridge $e'$ of $D$. 
The fact that every region of the dessin contains on its boundary essential vertices implies that after destroying the bridge $e$ the regions of the new graph have an oriented boundary with essential vertices. Therefore the resulting graph is a dessin and destroying that bridge is admissible.
All the elementary moves used to construct the elementary equivalent dessin $D'$ from $D$ are destructions of bridges. Since destroying bridges does not change the nature of essential vertices, $D$ and $D'$ have the same amount of inner essential vertices and real essential vertices.
\end{proof}

A real singular $\times$-vertex $v$ of index $2$ in a dessin can be perturbed within the class of real dessins 
on the same surface in two different ways. Locally, when $v$ is isolated,
the real part of the corresponding real trigonal curve has an isolated point as connected component of its real part, which 
can be perturbed to a topological circle or disappears, when $v$ becomes two real $\times$-vertices of index $1$ or one pair of complex conjugated interior $\times$-vertices of index $1$, respectively.
When $v$ is non-isolated, the real part of the corresponding real trigonal curve has a double point, which
can be perturbed to two real branches without ramification or 
leaving a segment of the third branch being one-to-one with respect to the projection $\pi$ while being three-to-one after two vertical tangents (see Figure~\ref{fig:deformsnodal}).  

\begin{figure}
\centering
    {\includegraphics[width=4in]{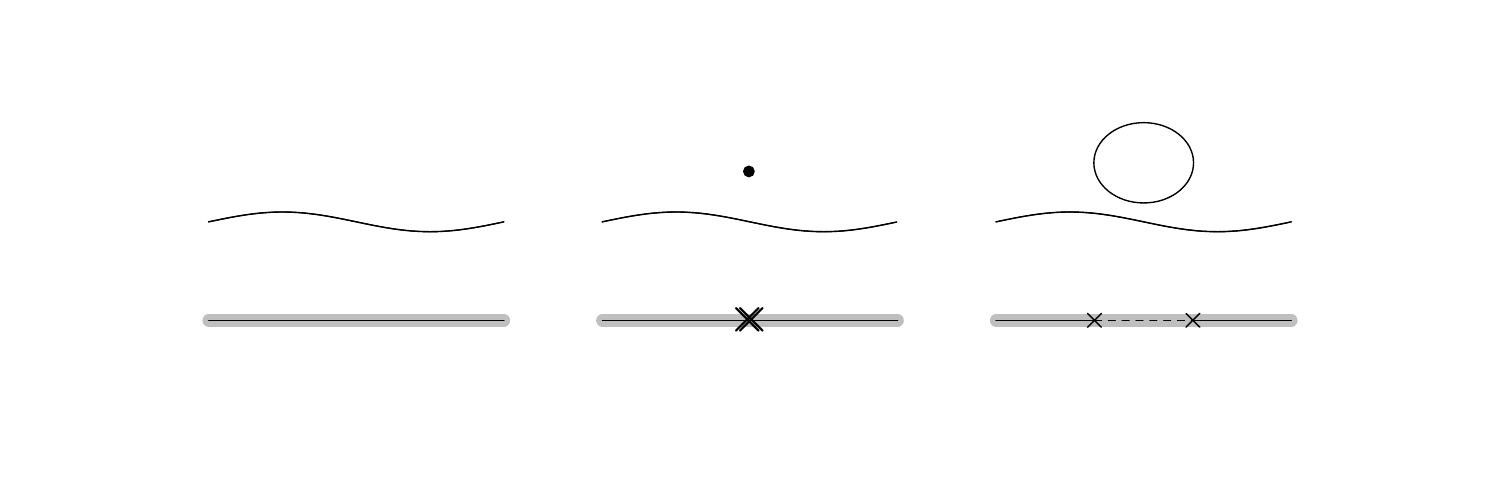}} \\
    {\includegraphics[width=4in]{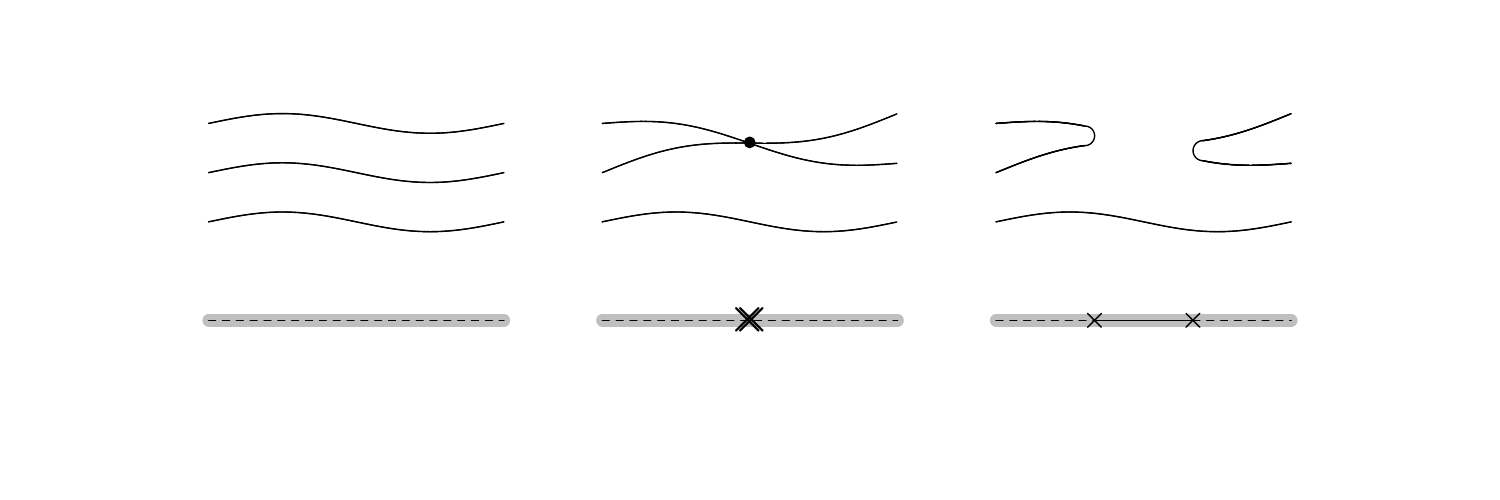}}
    \caption{Perturbations of a real $\times$-vertex of index $2$.}
    \label{fig:deformsnodal}
\end{figure}

\begin{df} 
Given a dessin $D$, a subgraph $\Gamma\subset D$ is a \emph{cut} if it consists of a single interior edge connecting two real monochrome vertices. An {\it axe} is an interior edge of a dessin connecting a real $\times$-vertex of index $2$ and a real monochrome vertex.
\end{df}
 
Let us consider a dessin $D$ lying on a surface $S$ having a cut or an axe~$T$. 
Assume that $T$ divides $S$ and consider the connected components $S_{1}$ and $S_2$ of~$S\setminus T$. 
Then, we can define two dessins $D_1$, $D_2$, each lying on the compact surface $\overline{S_i}\subset S$, respectively for $i=1, 2$, and determined by $D_i:=(D\cap S_i)\cup\{T\}$. 
If $S\setminus T$ is connected, we define the surface $S'=(S\setminus T)\sqcup T_1 \sqcup T_2/\varphi_1, \varphi_2$, where $\varphi_i:T_i\lra S$ is the inclusion of one copy $T_i$ of $T$ into $S$, and the dessin $D':=(D\setminus T)\sqcup T_1 \sqcup T_2/\varphi_1, \varphi_2$.
 
By these means, a dessin having a cut or an axe determines either two other dessins of smaller degree or a dessin lying on a surface with a smaller fundamental group. 
Moreover, in the case of an axe, the resulting dessins have one singular vertex less. 
Considering the inverse process, we call $D$ the \emph{gluing} of $D_1$ and $D_2$ along $T$ or the \emph{gluing} of $D'$ with itself along $T_1$ and $T_2$.


\begin{prop} \label{prop:peripheral}
A non-hyperbolic uninodal dessin is elementary equivalent to a peripheral one.
\end{prop}

\begin{proof}
We proceed by contradiction: let us assume that $D$ is a non-hyperbolic uninodal dessin, which is not elementary equivalent to a peripheral one. Let us choose a dessin $D'$ within the class of elementary equivalence of $D$ having a minimal number of inner essential vertices. By Corollary~\ref{cr:bf}, we can assume that $D'$ is a bridge-free dessin. 

We call an \emph{inner chain} a chain of vertices $v_0, v_1,\dots, v_k$ within $D'$ such that every edge $[v_i,v_{i+1}]$, $0\leq i<k$ and every vertex $v_i$, $0\leq i<k$ are inner. 
Let us assume that there is an inner chain $v_0, v_1,\ldots, v_k$ of minimal length starting at a vertex $v_0$, which is either white or black, arriving to a non-hyperbolic real component.

We pick the inner chain in a way that either the ending vertex $v_k$ is monochrome or every other inner chain of minimal length from a black or white vertex to the boundary ends in an essential vertex. This assumption and the fact that $D'$ is a bridge-free dessin assure the monochrome modifications we make are admissible. We study the possible configurations of the chosen inner chain in order to do elementary modifications decreasing the number of inner essential vertices.

 {\it Case 0:} $v_k$ is a monochrome vertex and $v_{k-1}$ is either black or white. Thus, the number of inner essential vertices is reduced after an elementary move of type $\bullet$-out or $\circ$-out.

 {\it Case 1.1:} $v_k$ is a black vertex and $v_{k-1}$ is a {\tv}. Then $k\geq 2$ and $v_{k-2}$ is white. Up to symmetry, the configuration corresponds to Figure \ref{fig4}. Hence, the creation of a bold bridge brings us to Case 0.

\begin{figure}[hbt]
\centering
\begin{subfigure}[b]{0.4\textwidth}
\centering
 \includegraphics[scale=0.4]{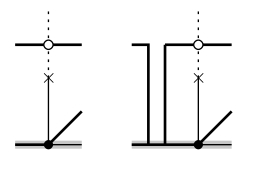}
\caption{}
\label{fig4}
\end{subfigure}
\begin{subfigure}[b]{0.4\textwidth}
\centering
 \includegraphics[scale=0.4]{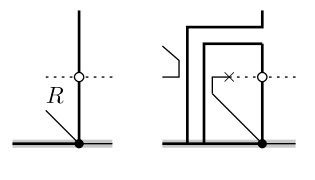}
\caption{}
\label{fig5}
\end{subfigure}
\begin{subfigure}[b]{0.4\textwidth}
\centering
 \includegraphics[scale=0.4]{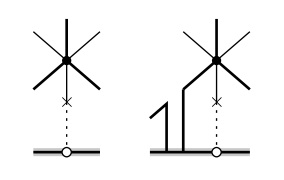}
\caption{}
\label{fig6}
\end{subfigure}
\begin{subfigure}[b]{0.4\textwidth}
\centering
 \includegraphics[scale=0.4]{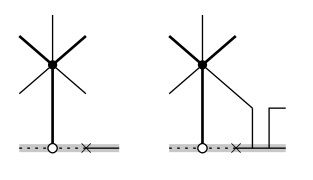}
\caption{}
\label{fig7}
\end{subfigure}
\caption{}
\end{figure}

\vskip5pt
 {\it Case 1.2.1:} $v_k$ is black and $v_{k-1}$ is white. We consider the region $R$ such that its boundary contains the edge $[v_{k-1},v_k]$ and the other inner edge of $v_k$. After one monochrome modification $R$ can be reduced to a triangle. If the region $R$ contains an inner {\tv} (see Figure \ref{fig5}), the creation of a bold bridge brings us to Case 0. 

\vskip5pt
 {\it Case 1.2.2:} $v_k$ is black and $v_{k-1}$ is white, and the above-mentioned region $R$ contains the nodal vertex $n$. By the minimality condition of the chain, the vertex $n$ is an elliptic nodal vertex (since $R$ is a triangle, and $v_{k-1}$ is not connected to a monochrome vertex) and the inner solid edge of $v_k$ connects it to a vertex $u$, at the boundary beside $n$. We look at the solid real edge of $v_k$, which connects $v_{k}$ with a vertex $v$, either an index $1$ {\tv} (since $D'$ is a uninodal dessin) or a monochrome vertex.
When the vertex $v$ is a {\tv}, there is a real dotted edge incident to $v$ where the construction of a dotted bridge brings us to Case $0$. 

 When the vertex $v$ is a monochrome vertex, let $e$ be its inner solid edge. 
This edge connects the boundary with either another monochrome vertex or with an inner {\tv}. If $e$ 
is adjacent to 
another monochrome vertex $w$, since $D'$ is a uninodal bridge-free dessin, the vertex $w$ belongs to a solid segment ending in two real {\tvs}, each defining a real dotted edge. Thus, creating a bridge with one of the inner dotted edge of $v_{k-1}$ at the real dotted edge in the same region brings us to Case $0$.
Otherwise, the edge $e$ connects $v$ to an inner {\tv} $w$. A monochrome modification defines an edge connecting $w$ to $v_{k-1}$. Since $D'$ is a bridge-free dessin, $v$ has two black neighbors. An elementary move of type $\bullet$-in at $v$, followed by an elementary move of type $\circ$-out at $u$, connects $v_{k-1}$ to the boundary at a monochrome vertex as in Case $0$.

\vskip5pt

 {\it Case 2.1:} $v_k$ is white and $v_{k-1}$ is a {\tv}. Hence, $k\geq 2$ and $v_{k-2}$ is black. Therefore, creating a bold bridge brings us to Case $0$ (see Figure \ref{fig6}).

\vskip5pt
 {\it Case 2.2:} $v_k$ is white and $v_{k-1}$ is black. We look at the neighboring vertices of $v_k$. If there is a {\tv} of index $1$, creating a solid bridge at the incident solid edge brings us to Case $0$ (see Figure \ref{fig7}). Otherwise, let $a$ be one of its monochrome neighbor such that the region $R$ containing the edges $[v_k,v_{k-1}]$ and $[v_k,v]$, as in the left part of Figure \ref{fig8},
does not contain the nodal vertex. Let $w$ be the white neighboring vertex of $a$ other than $v_{k}$. 
We replace the original chain by the chain $v_0, v_1,\dots, v_{k-1},w$. Since the real component is non-hyperbolic, after a finite number of iterations, the vertex $v_k$ will have a {\tv} of index $1$ as neighbor.

\vskip5pt
 {\it Case 3:} $v_k$ is monochrome and $v_{k-1}$ is a {\tv}. Then $k\geq2$ and $v_{k-2}$ is either white or black. We can perform a monochrome modification in order to connect $v_{k-2}$ with one of the real neighbors of $v_k$ and bring us to Case 1.2 (see Figure \ref{fig9}) or to Case 2.2 (see Figure \ref{fig10}).

\begin{figure}[hbt]
\centering
\begin{subfigure}[b]{0.45\textwidth}
\centering
 \includegraphics[scale=0.4]{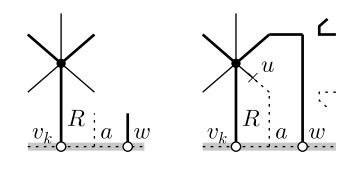}
\caption{}
\label{fig8}
\end{subfigure}
\begin{subfigure}[b]{0.45\textwidth}
\centering
 \includegraphics[scale=0.4]{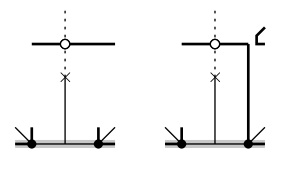}
\caption{}
\label{fig9}
\end{subfigure}
\begin{subfigure}[b]{0.4\textwidth}
\centering
 \includegraphics[scale=0.4]{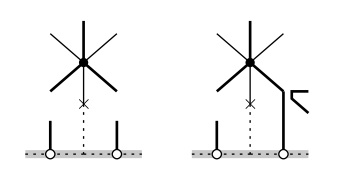}
\caption{}
\label{fig10}
\end{subfigure}
\caption{}
\end{figure}

\vskip5pt
 {\it Case 4:} $v_k$ is the nodal vertex. Let $R$ be one of the regions containing the edge $[v_{k-1},v_k]$. Let $e$ be the edge in the boundary of $R$, incident to $v_{k-1}$ and different from $[v_{k-1},v_k]$. If $v_{k-1}$ is black, the edge $e$ is bold and connects $v_{k-1}$ with a white vertex $v$. If $v$ is an inner white vertex, the creation of a dotted bridge, followed by an elementary move of type $\circ$-out reduces the number of inner essential vertices (see Figure \ref{fig01}). Otherwise, $v$ is a real white vertex. We consider instead the chain $v_0, v_1,\dots, v_{k-1},v$ as in Case 2.

Finally, if $v_{k-1}$ is a white vertex, the edge $e$ must connect $v_{k-1}$ to a black vertex~$v$. If $v$ is an inner black vertex, the creation of a solid bridge, followed by an elementary move of type $\bullet$-out reduces the number of inner essential vertices (see Figure \ref{fig02}). Otherwise, $v$ is a real black vertex. We consider instead the chain $v_0, v_1,\dots, v_{k-1},v$ as in Case 1.

In the case when all black or white inner vertices have no inner chain to a non-hyperbolic component, let us consider 
a black or white inner vertex $v_0$,
and let $v_0,v_1,\dots,v_k$ be a minimal inner chain connecting $v_0$ to a hyperbolic component. By the aforementioned considerations, the vertex $v_0$ must be black.
 In this setting, every hyperbolic component is connected to a non-hyperbolic component through a chain $C=\{u_0,u_1,\ldots, u\}$ free of inner black or white vertices. Up to monochrome modifications, the vertex $v_0$ belongs to a region whose 
boundary 
contains the inner chain $C$. Thus, the vertices $v_0$ and $u$ belong to the same region $R$. 

 {\it Case 5.1:} the vertex $u$ is monochrome. If it is solid or bold, the creation of a bridge on the real solid edge incident to $u$ in the region $R$ with an inner solid edge incident to $v_0$ brings us to Case 0. 
 If the vertex $u$ is dotted, it has a real neighboring {\tv} in the region $R$ determining a solid segment where the creation of a solid bridge with an inner solid edge incident to $v_0$ brings us to Case~0.

 {\it Case 5.2:} the vertex $u$ is nodal. Since the chain $C$ has no inner black vertices, $u$ has real solid edges. Thus, the creation of a solid bridge on the real solid edge incident to $u$ in the region $R$ with an inner solid edge incident to $v_0$ brings us to Case 0.
 
 {\it Case 5.3:} the vertex $u$ is white or black. If the vertex $u$ is white, it has real bold edges. Thus, the creation of a bold bridge on the real bold edge incident to $u$ in the region $R$ with an inner bold edge incident to $v_0$ brings us to Case 0. Otherwise, the vertex $u$ is black. 
 If in the region $R$ a bold real edge is incident to $u$, the creation of a bold bridge with an inner bold edge incident to $v_0$ brings us to Case~0. 
Otherwise, a solid monochrome modification between the inner solid edges incident to $u$ and $v_0$ in the region $R$ brings us to a configuration where the creation of a bold bridge on the real bold edge incident to $u$ with an inner bold edge incident to $v_0$ brings us to Case 0.
\end{proof} 

\begin{figure}[hbt]
\centering
\begin{subfigure}[c]{0.3\textwidth}
\centering
 \includegraphics[scale=0.6]{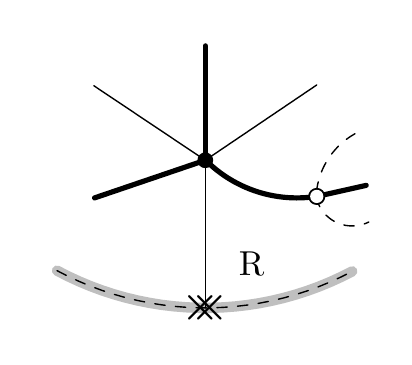}
\caption{}
\label{fig01}
\end{subfigure}
\begin{subfigure}[c]{0.3\textwidth}
\centering
 \includegraphics[scale= 0.6]{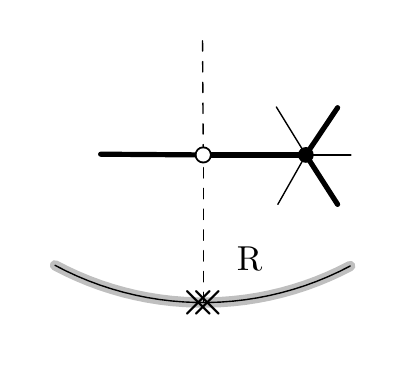}
\caption{}
\label{fig02}
\end{subfigure}
\caption{}
\end{figure}

\begin{prop} \label{prop:nowhite}
A hyperbolic uninodal dessin is elementary equivalent to a dessin without inner white vertices.
\end{prop}

\begin{proof}
By Corollary \ref{cr:bf}, we can suppose that the dessin is bridge-free. Let $v_0,\ldots,v_k$ be a minimal length chain from an inner white vertex $v_0$ to a real vertex $v_{k}$. The vertex $v_0$ is the only white inner vertex in the chain. Since every inner black or {\tv} is connected to a white vertex, the length $k$ is at most three. 

Since the dessin is hyperbolic, if $k=1$, then $v_{k}$ is a monochrome vertex. An elementary move of type $\circ$-out reduces the number of inner white vertices.

If $k=2$, then $v_{1}$ is black and $v_2$ is white or a nodal vertex. In the case when $v_2$ is 
white, we consider the solid edge $e$ incident to $v_1$ between the edges $[v_0,v_1]$ and $[v_1,v_2]$.

Let $v$ be the {\tv} of the edge $e$.
If $v$ is an inner {\tv}, a monochrome modification and the creation of a dotted bridge followed by an elementary move of type $\circ$-out reduces the number of inner white vertices (see Figure \ref{fig18}). 
Otherwise, the vertex $v$ is a nodal vertex and the creation of a dotted bridge followed by an elementary move of type $\circ$-out reduces the number of inner white vertices (see Figure \ref{fig01}).
When the vertex $v_2$ is nodal, we look at the white vertices connected to $v_1$. If the black vertex $v_1$ is connected to a real white vertex $v$, we consider the chain $v_0,v_1,v$ instead.

If $k=3$, the colors of the vertices $v_1,v_2,v_3$ are either $\times$, black and white, or black, $\times$ and monochrome, respectively. Due to the length minimality of the chain, in the former case every white vertex connected to $v_2$ is real, and the number of inner white vertices can be reduced by the creation of a dotted bridge followed by an elementary move of type $\circ$-out (see Figure \ref{fig19}); in the latter case every white vertex connected to $v_1$ is inner, thus any of the white vertices connected to $v_1$ within the region containing the edge $v_2, v_3$ becomes real after the creation of a dotted bridge and an elementary move of type $\circ$-out (see Figure \ref{fig20}), reducing the number of inner white vertices.
\end{proof}

\begin{figure}[htb]
\centering
\begin{subfigure}[b]{0.4\textwidth}
\centering
 \includegraphics[scale=0.4]{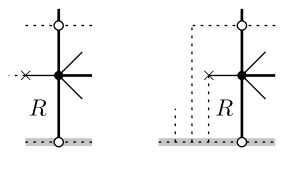}
\caption{}
\label{fig18}
\end{subfigure}
\begin{subfigure}[b]{0.4\textwidth}
\centering
 \includegraphics[scale=0.4]{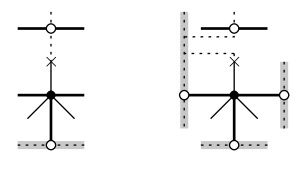}
\caption{}
\label{fig19}
\end{subfigure}
\begin{subfigure}[b]{0.4\textwidth}
\centering
 \includegraphics[scale=0.4]{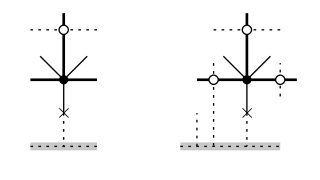}
\caption{}
\label{fig20}
\end{subfigure}
\caption{}
\end{figure}

\begin{thm} \label{prop:unide}
A non-hyperbolic uninodal real dessin $D$ on a surface $S$ of degree higher than 3 is weakly equivalent either to the gluing of dessins of smaller degree, or to the gluing of two real edges of a dessin $D'$. In both cases, the gluing corresponds to either a dotted cut or an axe. 
\end{thm}

\begin{proof}
Let $D_1$ be a dessin within the weak equivalence class of $D$, maximal with respect to the number of zigzags (equivalently, minimal with respect to the number of inner {\tvs}). The dessin $D_1$ is non-hyperbolic; therefore, by Proposition~\ref{prop:peripheral}, the dessin $D_1$ is elementary equivalent to a peripheral dessin $D_0$, which we can assume bridge-free, by Corollary~\ref{cr:bf}. 

A \emph{dotted cut} is an inner dotted edge $c$ connecting two monochrome vertices. If the dessin $D_0$ has a cut or an axe $c$, then cutting through the edge $c$ either defines two dessins $D_1$ and $D_2$ (in the case when $S\setminus c$ is disconnected), or it defines one dessin $D'$ on a surface~$S'$ (then $D_0$ is the image by the gluing of two different edges in $S'$). Let us assume $D_0$ has no dotted cut nor an axe. Let $T$ be the nodal vertex of $D_0$, let $e$ be the inner edge incident to $T$ and let $\mathcal{S}$ be the monochrome segment containing $T$. The edge $e$ connects $T$ with a real vertex $v$, which is either black (if $e$ is solid) or white (if $e$ is dotted). 

\vskip5pt
 {\it Case 1:} the edge $e$ is dotted and the vertex $v$ is white. Let us consider one region $R$ adjacent to $e$. In the case when the edge $e$ divides the surface $S$, let us assume that $R$ is on the connected component with a maximal number of black vertices. Let $u$ be the real neighbor vertex of $v$ in $R$.

\vskip5pt
 {\it Case 1.1:} the vertex $u$ is monochrome. Let us assume that $u$ has a real white neighbor vertex $v'$ different than $v$. The vertex $v'$ has an inner dotted edge $e'$. 
If the edge $e'$ connects $v'$ with a monochrome vertex $w$, an elementary move of type $\circ$-in at $u$ followed by an elementary move of type $\circ$-out at $w$ produces a dotted axe. Otherwise, the edge $e$ connects $v'$ with an inner {\tv} $w$.

If the vertex $u$ defines a bold cut, since the dessin is bridge-free, an elementary move of type $\circ$-in at $u$ followed by an elementary move of type $\circ$-out along the cut brings us to a configuration where, up to monochrome modifications, the vertex $w$ can be set up in order to create a zigzag, contradicting the maximality of the number of zigzags of $D_0$.

Otherwise, the vertex $u$ is connected to a real black vertex $w'$. If the vertex $w'$ has a real neighbor {\tv} $w''$, we can create a dotted axe, either by creating a dotted bridge with the edge $e$ if $w''$ belongs to $R$ or by creating a dotted bridge with the edge $e'$ beside $w''$ and proceeding as $v'$ is connected to a monochrome vertex, if the vertex $e'$ is contained on the region containing $w''$. If the vertex $w'$ is neighbor to a monochrome vertex $w''$, an elementary move of type $\bullet$-in at $w''$ followed by an elementary move of type $\bullet$-out at $u$ bring us to a setting around $w$ where we can create a zigzag.

If the vertex $w'$ is connected to the nodal vertex $T$, up to a solid monochrome modification, the vertices $w$ and $w'$ are connected. In this case, an elementary move of type $\circ$-in, followed by the creation of a bold bridge beside $v'$ and an elementary move of type $\circ$-out bring us to a setting around $w$ where we can create a zigzag.

One special case is when $v$ belongs to a monochrome boundary component, having a bold monochrome vertex $u$ with inner edge $e'$. If the vertex $u$ is connected by $e'$ to a black vertex $w$ neighbor to a real {\tv} determining a dotted segment where the creation of a bridge with $e$ produces an axe. 
If the vertex $u$ is connected by $e'$ to a monochrome vertex~$w$ or to a black vertex neighbor to a solid monochrome vertex $w'$, then an elementary of type $\bullet$-in at $w$ or $w'$ followed by an elementary move of type $\bullet$-out at $u$ brings us to the configuration where $v$ has to black real neighbors $u_1$, $u_2$, they three being the only essential vertices of that connected component of the boundary.

If one of the black vertices $u_1$ or $u_2$ is connected to an inner {\tv} $w$, a dotted monochrome modification connects $w$ with $v$ and the creation of a bold bridge brings us to a setting around $w$ where we can create a zigzag. 

If one of the black vertices $u_1$ or $u_2$ is connected to a solid monochrome vertex $w$ not in the segment $\mathcal{S}$, its real neighbor on the region $R$ is a {\tv} since the dessin is bridge-free, determining a dotted segment where the creation of a bridge with $e$ produces an axe.

In the remaining case, both black vertices $u_1$ and $u_2$ are connected to solid monochrome vertices $w_1$ and $w_2$, respectively, neighbors of the nodal vertex $T$. 
Since the dessin is bridge-free, each $w_i$ ($i=1,2$) is neighbor to a non-singular {\tv} $w_{i}'$, respectively. If a vertex $u_i$ ($i=1,2$) is connected to a bold monochrome vertex, an elementary move of type $\circ$-in at it followed by the creation of a dotted bridge beside $w_{i}'$ (if needed) and an elementary move of type $\bullet$-out set the vertex $u_i$ connecting to a white vertex $w_i ''$ neighbor to $w_i '$, respectively. 
If a vertex $u_i$ ($i=1,2$) is connected to a white vertex and the vertex $w_i '$ is connected to a dotted monochrome vertex, the creation of a dotted bridge after a bold monochrome modification produces a cut. 
In any case, up to a bold monochrome modification, we can assume each vertex $u_i$ connected to a white vertex $w_i ''$ neighbor to~$w_i '$, respectively.

Let $u'$ be the solid monochrome vertex neighbor to $u_1$ and $u_2$. If the vertex $u'$ is connected to an inner {\tv}, then the creation of dotted bridges (if needed) at the dotted segments determined by $w_i ''$ produces a cut. 
Otherwise, the vertex $u'$ determined a solid cut, with monochrome vertex $u''$ having as neighbor {\tvs} $w_i '''$. 
If one of the {\tvs} $w_1'''$ or $w_2'''$ is neighbor to a monochrome vertex, the creation of a dotted bridge beside $w_i ''$ produces a cut. Otherwise, each $w_i'''$ is neighbor to a white vertex $v_i$. It can happen that $v_i=w_i''$ for one $i=1,2$ but not for both since the degree is greater than $3$. Let $i_1$ be an index for which $v_{i_1}\neq w_{i_1}''$ and let $i_0\in\{1,2\}\setminus\{i_1\}$. If the vertex $v_{i_0}$ is neighbor to a dotted monochrome vertex, the creation of a bridge beside $ w_{i_0}''$ produces a cut. Otherwise, $v_{i_0}$ is neighbor to a {\tv} determining a zigzag. We do a bold monochrome modification so $u_{i_0}$ is connected to $v_{i_0}$. Then, an elementary move of type $\bullet$-in at $u'$ followed by an elementary move of type $\bullet$-out through the solid cut and the creation of a bridge bring us to a setting around $v_{i_0}$ where we destroy the zigzag and the remaining bridge. The resulting configuration has an inner {\tv} having a dotted edge. The creation of a dotted bridge beside $v_{i_1}$ followed by the creation of a bridge beside $w_{i_1}''$ produces a cut.

Otherwise, the vertex $u$ is connected by $e'$ to a black vertex $w$ neighbor of $T$. In this setting, we look at the inner solid edge $f$ of $w$. 
If the edge $f$ connects $w$ with a monochrome vertex $w'$, which has a {\tv} $w''$ neighbor in the region determined by $e$ and $f$ since the dessin is bridge-free, then the creation of a bridge with $e$ on the dotted segment determined by $w''$ produces an axe.
If the edge $f$ connects $w$ with an inner {\tv} $w'$, we look at the real neighbor vertex $w''$ of $w$. If the vertex $w''$ is white, up to a monochrome modification it is connected to $w'$ and the creation of a bold bridge bring us to a setting around $w'$ where we can create a zigzag. 
If the vertex $w''$ is monochrome, then an elementary move of type $\bullet$-in at $w''$, followed by an elementary move of type $\bullet$-out and a dotted monochrome modification to connect $v$ with $w'$ bring us to a setting around $w'$ where we can create a zigzag.

\vskip5pt
{\it Case 1.2:} the vertex $u$ is black. 
Let $f$ to be the inner solid edge of $u$ and let $w$ to be the vertex connected to $u$ by $f$. If the vertex $w$ is monochrome and neighbor to an index $1$ {\tv} on the region $R$, it determines a dotted segment in which the creation of a bridge with $e$ produces an axe. 
 If the vertex $w$ is monochrome and neighbor to the nodal vertex $T$, it has a real neighbor {\tv} $w'$. Then, we ask for the real neighbor vertex $w''$ of $u$. If the vertex $w''$ is monochrome, an elementary move of type $\bullet$-in at $w''$ followed by an elementary move of type $\bullet$-out at $w$ and the destruction of a possible bridge, repeated a finite number of times, bring us to Case~{1.1} or a considered case.

 If the vertices $w'$ and $w''$ are index $1$ {\tvs}, we consider the bold inner edge $f'$ of~$u$. 
 If the edge $f'$ connects $u$ to a monochrome vertex within a monochrome component with on essential vertex $u'$, then the creation dotted bridges with the inner edge of $u$ beside $w'$ and $w''$ produces a dotted cut. It can not happen that the vertices $w'$ and $w''$ are real neighbors of the same dotted monochrome vertex since the degree of the dessin is greater than $3$. 
 If the edge $f'$ connects $u$ to a monochrome vertex $w'''$ having two real neighbors, then an elementary move of type $\circ$-in at $w'''$ followed by the creation of a dotted bridge beside $w'$, an elementary move of type $\circ$-out and the creation of a bridge beside $w''$ produce a dotted cut. 

If the edge $f'$ connects $u$ to a white vertex $u'$, then the vertex $u'$ cannot be a real neighbor of $w'$ and $w''$ since the degree is greater than $3$. If the vertex $w''$ has a real neighbor monochrome vertex, then the creation of a dotted bridge beside $u'$ produces a dotted cut. Otherwise we can assume the vertex $u'$ is a real neighbor of $w''$, up to a monochrome modification. If one of the vertices $u'$ and $w'$ has a real neighbor monochrome vertex, then the creation of a dotted bridge produces a dotted cut.
If in the region determined by the edges $f$ and $f'$ there is a dotted inner edge, up to the creation of bridges beside the vertices $u'$ and $w'$ there exists a dotted cut. If the vertex $u'$ has a real neighbor {\tv} connected to a solid cut, an elementary move of type $\bullet$-in followed by an elementary move of type $\bullet$-out bring us to the case where the vertex $u'$ has a real neighbor {\tv} connected to a real black vertex $b$. Constructing a solid bridge with the edge $f$ beside the vertex $b$ results in a solid cut and allows us to destroy the zigzag determined by $u'$. An elementary move of type $\bullet$-in followed by an elementary move of type $\bullet$-out along the cut, elementary moves of type $\circ$-in and $\circ$-out and the creation of a zigzag bring us to a configuration where we can restart the algorithm. Since the degree is greater than~$3$ and the dotted edge of the nodal vertex $T$ is dividing, we do not obtain a cyclic argument.

If the vertex $w'$ is equal to the nodal vertex, then the creation of a dotted bridge with the edge $e$ beside the vertex $u'$ results in the creation of a dotted axe.

If the vertex $w''$ is equal to the nodal vertex, we look at the inner bold edge $f'$ of~$u$. 
If the edge $f'$ connects the vertex $u$ to a white vertex belonging to a dotted segment where the creation of a dotted bridge with the edge $e$ creates an axe.
If the edge $f'$ has a monochrome vertex, then there is a white vertex $b$ in the region determined by the edges $f$ and $f'$. If the inner dotted edge of the vertex $b$ connects it to a monochrome vertex, then we can create an axe by constructing a bridge with the edge $e$ by either eluding the monochrome boundary component containing $b$ or by performing elementary moves of type $\circ$-in and $\circ$-out on the dotted segment.
If the inner dotted edge of the vertex $b$ connects it to a {\tv} vertex, monochrome modifications to join it to the vertices $u$ and $v$ and the creation of a bold bridge allow us to create a zigzag, which implies that the dessin $D_0$ has not a maximal number of zigzags.


\vskip5pt
 {\it Case 2:} the edge $e$ is solid and the vertex $v$ is black.

\vskip5pt
 {\it Case 2.1:} the vertex $v$ has a real bold edge connecting to a white vertex $u$, 
having an inner dotted edge $f$. The edge $f$ cannot connect $u$ to an inner {\tv}, since a monochromatic modification would allows us to create a zigzag. Since the dessin is peripheral, the edge $f$ connects $u$ to a monochrome vertex $w$. If the vertex $w$ is not a real neighbor of the nodal vertex $T$, the creation of a bridge with the edge $f$ beside $T$ produces a dotted cut. Otherwise, the vertex $w$ is a real neighbor of the vertex $T$.
The vertex $u$ has a real neighbor vertex $v'$. If the vertex $v'$ is black, its inner solid edge $e'$ must be incident to a monochrome vertex neighbor to two real {\tvs} $w'$ and $w''$. Let us assume the vertex $w'$ belongs to the region $R'$ containing the edges $f$ and $e'$. Since the dessin is uninodal, the {\tv} $w'$ determines a dotted segment, at which the creation of a bridge produces a dotted cut unless the vertices $w$ and $w'$ are connected. 

In the case where the vertices $w$ and $w'$ are real neighbors, let $f'$ be the inner bold edge incident to $v'$, connecting it to a real vertex $u'$. 
Let us assume the vertex $u'$ is white. 
If the region determined by the edges $e'$ and $f'$ contains an inner dotted edge, up to the creation of bridges beside the vertices $w''$ and $u'$ we can obtain a dotted cut. Otherwise, up to a monochrome modification the region $R'$ is triangular. Let $w'''$ be the real neighbor vertex of $v'$. If the vertex $w'''$ is a {\tv} followed by a dotted monochrome vertex, the creation of a bridge beside the vertex $u'$ produces a dotted cut. 

It can occur the vertex $u'$ is neighbor to the vertices $w''$ and $w'''$. We will analyze this configuration as Case~2.3.


If the case when the vertex $w'''$ is a {\tv} followed by a white vertex, we can assume that $w'''$ and $u'$ are connected. 
If the region determined by the edges $e'$ and $f'$ contains an inner dotted edge, up to the creation of bridges beside the vertices $w''$ and $u'$ we can obtain a dotted cut. Otherwise, we look at the sequence of vertices on the boundary after $w'''$ and $u'$. If the vertex $u'$ is followed by a {\tv} and a solid monochrome vertex, it must be followed by a {\tv} $b$ since the dessin is bridge-free and uninodal. Up to a monochrome modification, we can construct a bridge with the edge $f$ beside the vertex $b$ producing a cut. If the vertex $u$ is followed by a {\tv} and a black vertex $b$, having an inner solid edge connecting to a monochrome vertex $m$. 
The destruction of the zigzag at~$u$, elementary moves of type $\bullet$-in and $\bullet$-out, and the construction of a bridge bring us to a configuration where the creation of a bridge beside $m$ with the edge $e$ produces an axe. 
If the vertex $b$ is connected by its inner solid edge to an inner {\tv} $m$, after monochrome modifications we can construct bridges with the inner dotted edge of $m$ beside $w'''$ and the real white neighbor of $w'''$ in order to obtain a cut.

In the case when the vertex $w'''$ is monochrome, we perform elementary moves of type $\bullet$-in at $w'''$ and $\bullet$-out along $e'$, destroying the possible residual bridge. This corresponds to a configuration already considered.

We proceed in the same way if the edge $f'$ connects the vertex $v'$ to a monochrome vertex $m$ and $v'$ has a monochrome solid neighbor. 

Otherwise $v'$ has a real {\tv} neighbor $w'''$ determining a dotted segment. 
Then, either the vertex $m$ has two real white vertices as neighbors, an elementary move of type $\circ$-in at $m$ followed by an elementary move of type $\circ$-out (up to the creation of a bridge beside $w'''$) and the creation of a bridge beside $w''$ produces a dotted cut; or the vertex~$m$ belongs to a monochrome component, having a white vertex with inner dotted edge, with which the creation of two dotted bridges beside $w''$ and $w'''$ produces a cut.

Finally, when the vertex $v'$ is monochrome, an elementary move of type $\circ$-in at $v'$ followed by an elementary move of type $\circ$-out at $w$ and the destruction of a possible residual bold bridge bring us to a configuration to which we iterate this process. Due to the finite number of vertices this process must stop.

\vskip 5pt
 {\it Case 2.2:} the vertex $v$ has a real bold edge connecting to a monochrome vertex $u$, 
 having an inner bold edge $f$ connecting $u$ to a vertex $w$. When the vertex $w$ is monochrome, an elementary move of type $\bullet$-in at $u$ followed by an elementary move of type $\bullet$-out at $w$ bring us to Case~$2.1$.

If the vertex $w$ is white and it is neighbor to a monochrome or simple {\tv} in the region determined by the edges $e$ and $f$, then the creation of a dotted bridge beside $T$ produces a cut or a solid bridge beside the {\tv} produces an axe, respectively. Otherwise, $w$ is a real neighbor of $T$. Since the dessin is bridge-free, the vertex $u$ has a real neighbor $v'\neq v$ having an inner bold edge $f'$. If the edge $f'$ connects $v'$ to a monochrome vertex, an elementary move of type $\bullet$-in at $u$ followed by an elementary move of type $\bullet$-out along $f'$ produces a solid axe.

Let us assume the edge $f'$ connects $v'$ to a white vertex $w'$. Let $e'$ be the solid inner edge incident to $v'$. If the edge $e'$ connects $v'$ to an inner {\tv} $w''$ having an inner dotted edge $g$, then the creation of dotted bridges (if needed) beside them with $g$ produces a dotted cut, unless the vertices $w$ and $w'$ are real neighbors to a monochrome vertex incident to $g$, at which an elementary move of type $\circ$-in followed by an elementary move of type $\circ$-out at $u$ brings us to a configuration where we can construct a zigzag determined by $w''$.

If the edge $e'$ connects $v'$ to a monochrome vertex $w''$, an elementary move of type $\circ$-in at $u$ followed by an elementary move of type $\circ$-out at $w''$ brings us to a configuration where we can iterate the process. If it cycles to this configuration, we look at the sequence of real neighbor vertices of $w''$ in the region determined by the edges $f'$ and $e'$. If the vertex $w''$ if followed by a {\tv} and a monochrome vertex, or by a {\tv}, a white vertex $w'''$ and a monochrome vertex, then the creation of a dotted bridge beside $w'$ or the monochrome modification to connect the vertices $v'$ and $w'''$ followed by the creation of a dotted bridge produce a cut. This move is admissible since otherwise $w''$ and $w'''$ would be the same and the degree of the dessin would be $3$. If the vertex $w''$ is followed by a zigzag, we do the same consideration in this paragraph with respect to the vertex $v'$.

Finally, in the case where the vertices $v'$ and $w''$ are followed by zigzags with white vertices $w_1$ and $w_2$, let us assume $v'$ to be connected to $w_1$, the white vertex determining its adjacent zigzag. Let $g$ be the inner bold edge incident to $w_2$ and let $m$ be its other vertex.
If the vertex $m$ is monochrome, the creation of bold bridge with the edge $f'$ beside $m$ allows us to proceed as when the edge $f'$ connected $v'$ to a monochrome vertex.
If the vertex $m$ is black, being connected to an inner solid edge incident to an inner {\tv} $m'$, then up to a monochrome modification between $f'$ and $g$, the vertex $m'$ belongs to the same region that $w_1$ and $w_2$. Thus, the creation of dotted bridges with the inner edge incident to $m'$ beside $w_1$ and $w_2$ produces a cut.

If the vertex $m$ is black being connected to a monochrome solid vertex $m'$, which up to a monochrome modification between $f'$ and $g$ belongs to the same region as $w_1$ and $w_2$, then a monochrome modification between the solid inner edges of $v'$ and $m$ brings us to a configuration already considered. 

\vskip5pt
 {\it Case 2.3:} the edge $e$ does not divide the dessin into two parts, one of these consisting of two triangles. 
\vskip5pt
 {\it Case 2.3.1:} the vertex $v$ has a bold inner incident edge $e'$ connecting $v$ to a white vertex~$v'$.
If within the region determined by the edges $e$ and $e'$ the vertex $T$ has a monochrome real neighbor, having an inner dotted edge $f$, we can create a bridge beside $v'$ in order to obtain a cut. Otherwise, the vertex $T$ has a white real neighbor, which up to a monochrome modification is the vertex $v'$. Let us assume the vertex $v'$ has a real neighbor vertex $u$ which is monochrome.

We ask for the color of the real neighbor vertex $w$ of $v$. If $w$ is a {\tv}, then the creation of a dotted bridge with the inner edge $f$ of $u$ beside $w$ forms a cut, as in Figure~\ref{fig111a}. 
Otherwise $w$ is a solid monochrome vertex, either defining a solid cut where $w$ is connected to a segment bounded by {\tvs}, which allows us to create a bridge with $f$ producing a cut, or $w$ is connected to an inner {\tv}, and after a triangulation, the configuration of the dessin is as in Figure~\ref{fig111b}. 

The vertex $w$ has a black neighbor $v''$ with an inner solid edge $g$ incident to a vertex $w'$. If the vertex $w'$ is monochrome then an elementary move of type $\bullet$-in on $w$ followed by an elementary move of type $\bullet$-out set us as in Figure~\ref{fig111c}. 
In this case, an elementary move of type $\circ$-in followed by an elementary move of type $\circ$-out  allows us to create a zigzag
as in Figure~\ref{fig111d}. 

Otherwise, the vertex $w'$ is an inner {\tv} connected to a monochrome vertex $w''$, which can be on the segment $S$ or on a different segment. In the former case, the creation of a dotted bridge on $S$ with the inner dotted edge of $w'$ produces a dotted cut. In the latter case, $w''$ has two white real neighbor vertices and, up to the creation of a bold bridge, an elementary move of type $\circ$-in on $w''$ followed by an elementary move of type $\circ$-out on the bold segment of $v'$ 
allow us to create a zigzag, contradicting the maximality of the number of zigzag of the dessin (see Figure~\ref{fig111e}).

In the case where the vertex $w'$ is an inner {\tv} connected to a white vertex, up to the creation of a bold bridge, an elementary move of type $\bullet$-in, followed by an elementary move of type $\bullet$-out produces a solid axe, as in Figure~\ref{fig111f}. 

Finally, the vertex $w'$ can be a solid monochrome vertex. Since the dessin is bridge-free, the vertex $w'$ has two real neighboring {\tvs}. The creation of a dotted bridge with the edge $f$ beside the {\tv} in its region produces a dotted cut.

Let us assume now that the vertex $u$ is a {\tv}. By the assumption on the present Case, it cannot be connected to the vertex $v$.
Let $w$ be the real neighbor vertex of $v$ on the region containing $e'$.
If the vertex $w$ is monochrome, having an inner solid edge $f$, the creation of a bridge with $f$ beside $u$, followed by elementary moves of type $\bullet$-in at $w$, $\bullet$-out along $f$ and the destruction of a possible residual bold bridge bring us to a configuration where we can iterate this process, considering that the new edge $e$ does divide the dessin, with one connected component being formed by two triangles.

In the case when the vertex $w$ is a {\tv}, neighbor of a dotted monochrome vertex, this one has an inner dotted edge, with which the creation of a bridge beside $u$ produces a cut. 
Instead, if the vertex $w$ is a {\tv}, neighbor of a white vertex $w'$, then the vertex $w'$ has an inner bold edge $f$ connecting it to a black vertex $w''$.
If the inner solid edge $f'$ of $w''$ belong to the same region of $e'$, a bold monochromatic modification between $e'$ and $f$ changes the region so the creation of a solid bridge with $e$ beside $w''$ produces an axe.
Otherwise, let $w'''$ be the incident vertex of $f'$ other than $w''$. 
If the vertex $w'''$ is a {\tv}, it has an inner dotted edge, with which the the creation of bridges beside the vertices $u$ and $w''$ after the bold monochrome modification between the edges $e'$ and $f$ produces a cut.
If the vertex $w'''$ is monochrome, the creation of a solid bridge with the edge $e$ beside $w'''$ produces a cut.

%

\begin{figure}[htb]
\centering
\begin{subfigure}[b]{0.4\textwidth}
\centering
 \includegraphics[scale=0.7]{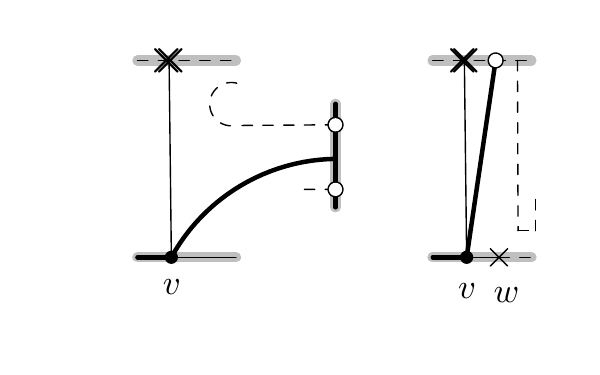}
\caption{}
\label{fig111a}
\end{subfigure}
\begin{subfigure}[b]{0.3\textwidth}
\centering
 \includegraphics[scale=0.7]{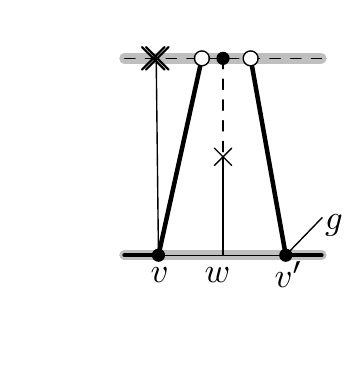}
\caption{}
\label{fig111b}
\end{subfigure}
\begin{subfigure}[b]{0.35\textwidth}
\centering
 \includegraphics[scale=0.7]{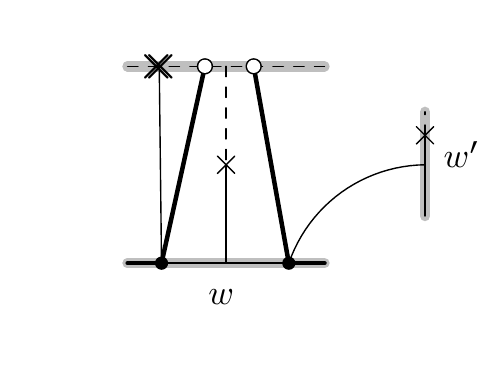}
\caption{}
\label{fig111c}
\end{subfigure}
\begin{subfigure}[b]{0.5\textwidth}
\centering
 \includegraphics[scale=0.7]{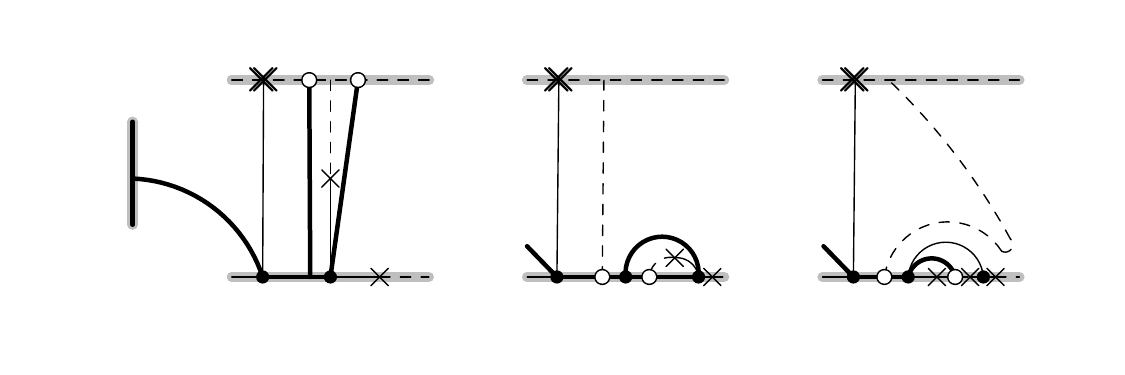}
\caption{}
\label{fig111d}
\end{subfigure}
\begin{subfigure}[b]{0.4\textwidth}
\centering
 \includegraphics[scale=0.6]{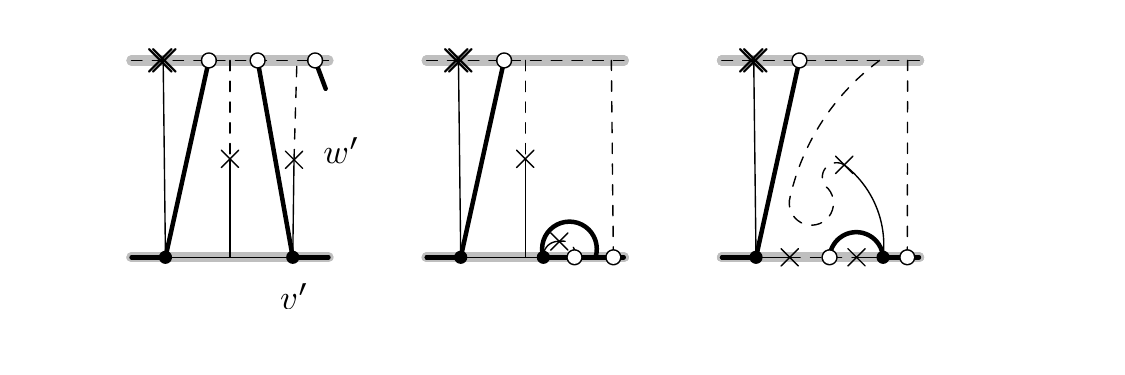}
\caption{}
\label{fig111e}
\end{subfigure}
\begin{subfigure}[b]{0.5\textwidth}
\centering
 \includegraphics[scale=0.6]{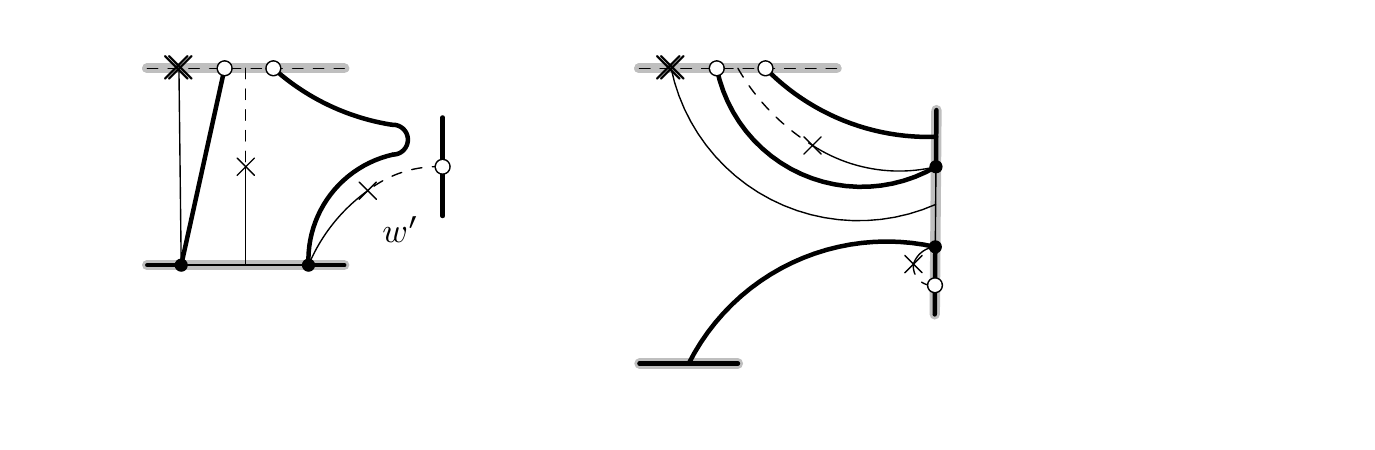}
\caption{}
\label{fig111f}
\end{subfigure}
\begin{subfigure}[b]{\textwidth}
\centering
 \includegraphics[scale=0.6]{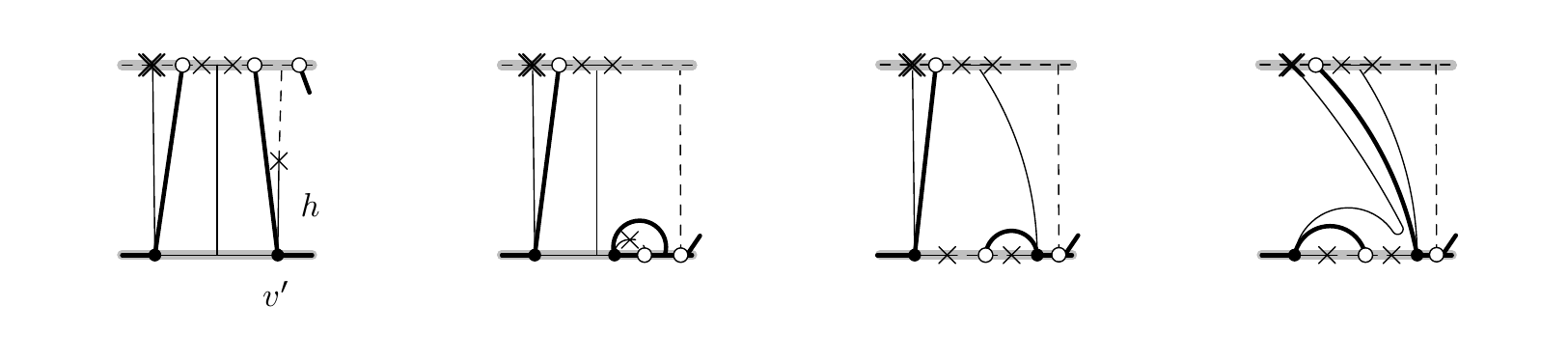}
\caption{}
\label{fig111g}
\end{subfigure}
\caption{}
\end{figure}

\vskip5pt
{\it Case 2.3.2:} the vertex $v$ has a bold inner incident edge $e'$ connecting $v$ to a monochrome vertex $v'$. 
Assume that the vertex $v'$ has two real neighboring white vertices. Then, the region containing the edges $e$ and $e'$ has an inner dotted edge $f$ which, up to the creation of a bridge, is incident to a monochrome vertex neighbor to the nodal vertex $T$. 
An elementary move of type $\circ$-in on $v'$ followed by an elementary move of type $\circ$-out on the monochrome of the edge $f$ brings us to Case 2.3.1, as seen in the see Figure~\ref{fig111a}.

If the vertex $v$ has a bold inner incident edge $e'$ connecting $v$ to a monochrome vertex $v'$ belonging to a bold monochrome component with only one white vertex $b$ having an inner dotted edge $f$. Then we look at the real neighbor $w$ of $v$ by a solid edge. Either the vertex $w$ is a {\tv} and the construction of bridges (if necessary) with the edge $f$ beside $w$ and $T$ produces a cut, or the vertex $w$ is monochrome and an elementary move of type $\bullet$-in at $w$ followed by an elementary move of type $\bullet$-out at $v'$ bring us to Case~2.1.
\end{proof}


%
%
%
%
%
%
%

\begin{prop} \label{prop:dechyp}
A hyperbolic uninodal real dessin $D$ of degree higher than 3 is weakly equivalent either to the gluing of dessins of smaller degree, or to the gluing of two real edges of a dessin $D'$. In both cases, the gluing corresponds to a dotted cut. 
\end{prop}

\begin{proof} 
Let $D$ be a hyperbolic uninodal dessin and let $T$ be the nodal vertex of $D$. Due to Corollary \ref{cr:bf} and Proposition \ref{prop:nowhite} we can assume that the dessin $D$ is bridge-free and has no inner white vertex.
The vertex $T$ is connected to an inner black vertex $b$, which is connected to inner {\tvs} $v_1$ and $v_2$, each connected to a monochrome vertex $u_1$ and $u_2$, respectively.

\vskip 5pt
{\it Case 1:} the vertices $u_1$ and $u_2$ have a common real neighboring white vertex. 
Since the degree is greater than $3$ at least one of the vertices $u_1$ or $u_2$ has no common real neighboring vertex with~$T$.
Let us assume $u_1$ has no common real neighboring vertex with $T$.
An elementary move of type $\circ$-in at $u_1$ produces an inner white vertex $w$ which up to a monochrome modification is connected {\it via} two different inner bold edges to $b$. 
In this setting, the creation of a dotted bridge near $T$ with an inner dotted edge of $w$ followed by an elementary move of type $\circ$-out sending $w$ next to the vertex $T$ produces a dotted cut in a dessin without inner white vertices.

\vskip 5pt
{\it Case 2:} the vertices $u_1$ and $u_2$ have no common real neighboring vertex. 

{\it Case 2.1:} at least one of the vertices $u_1$ or $u_2$ has two real white neighboring vertices.
Let us assume that $u_1$ has two real white neighboring vertices.
An elementary move of type $\circ$-in at $u_1$ produces an inner white vertex $w$ which up to a monochrome modification is connected {\it via} two different inner bold edges to $b$. 
In this setting, the creation of a dotted bridge near $u_2$ with an inner dotted edge of $w$ followed by an elementary move of type $\circ$-out sending $w$ next to the vertex $u_2$ produces a dotted cut in a dessin without inner white vertices.

{\it Case 2.2:} every vertex $u_i$, $i=1,2$, belongs to a component with only one essential vertex. Let $w$ be the white vertex in the real component of $u_1$. Up to a monochrome modification the vertex $w$ is connected to $b$ such that $b$, $v_1$, $w$ form a triangle. 
In this setting, the vertices $v_1$ and $u_2$ belong to the same region. Creating a dotted bridge near $u_2$ with the inner dotted edge of $v_1$ produces a dotted cut.
\end{proof}



\section{Real pointed quartic curves}\label{ch:hir}
\label{sec:realplane}

\subsection{Hirzebruch surfaces}

From now on we consider ruled surfaces with the Riemann sphere $B\cong\CP$ as base curve. The Hirzebruch surfaces $\Sigma_n$ are toric surfaces geometrically ruled over the Riemann sphere, determined up to isomorphism of complex surfaces by a parameter $n\in\N$. They are minimal except for $\Sigma_1$, which is isomorphic to $\CPP$ blown up in a point. The surface $\Sigma_n$ is defined by the local charts $U_0\times\CP$ and $U_1\times\CP$ where $U_0=\{[z_0:z_1]\in\CP\mid z_0\neq0\}$, $U_1=\{[z_0:z_1]\in\CP\mid z_1\neq0\}$ glued along $\C^*$ via the map 
\[
\displaystyle \begin{array}{ccc} \C^*\times\CP & \lra & \C^*\times\CP \\ {(z,w)} & \lmt & \left(\frac{1}{z}, \frac{w}{z^n}\right) \end{array}.
\]

The exceptional section is the section at infinity $E$ such that $S_{\infty}^2=-n\,(n\geq0)$. The second homology group $H_2 (\Sigma_n,\ZZ)$ of the Hirzebruch surface $Sigma_n$ is generated by the homology classes $[Z]$ and $[f]$ of the null section $Z$ and of a fiber $f$, respectively. The intersection form is determined by the Gram matrix $$\left( \begin{array}{cc} n & 1 \\   1 &  0 \end{array}  \right)$$ with respect to the base $\{[Z],[f]\}$. The homology class of the exceptional section is given by $[E]=[Z]-n[f].$
Performing a positive Nagata transformation on $\Sigma_n$ results in a geometrically ruled surface isomorphic to $\Sigma_{n+1}$ (since the exceptional divisor decreases its self-intersection by one). Likewise, a negative Nagata transformation on $\Sigma_n\,(n>0)$ results in a geometrically ruled surface isomorphic to $\Sigma_{n-1}$.

Setting $(\C^*)^2\subset\Sigma$, the trigonal curve $C$ can be described by a polynomial in two variables $f(z,w)=q_0(z)w^3+q_1(z)w^2+q_2(z)w+q_3(z)$ in which $q_0$ determines the intersection with the exceptional fiber. If the trigonal curve $C$ is proper, then $q_0$ must be constant. We can suppose $q_0$ equal to $1$. Up to affine transformations of $\C$, we can set the sum of the roots of $f(z,\cdot)$ equal to $0$, resulting in the Weiertra{\ss} equation $$w^3+q_2(z)w+q_3(z).$$

Since $C$ is a trigonal curve, then $[C]=3[Z]+0[f]$. Therefore, the intersection product $[C]\cdot[Z]=3n$ equals the degree of $q_3(z)$. Since this explicit description must be invariant by the coordinate change ${(z,w)} \lmt (\frac{1}{z}, \frac{w}{z^n})$, the degree of $q_2$ must be $2n$. Hence, the $j$-invariant $$\displaystyle j=\frac{-4q_2^3}{\Delta}, \;\Delta=-4q_2^3-27q_3^2$$ is a rational function of degree $6n$ if the curve is generic (i.e., the polynomials $q_2$ and $q_3$ have no zeros in common).

\subsubsection{Relation with plane curves}

Let $A\subset\CPP$ be a reduced algebraic curve with a distinguished point of multiplicity $\deg(A)-3$ such that $A$ does not have linear components passing through $P$. The blow-up of~$\CPP$ at~$P$ is isomorphic to $\Sigma_1$. The strict transform of $A$ is a trigonal curve $C_{A}:=\widetilde{A}\subset\Sigma_1$, called the \emph{trigonal model} of the curve $A$. A \emph{minimal proper model} of $A$ consists of a proper model of $C_{A}$ and markings corresponding to the images of the improper fibers of $C_{A}$ by the Nagata transformations.

\subsubsection{Real algebraic plane curves of degree 3}\label{sec:cubics}

\begin{figure}
\centering
\begin{tabular}{cc}
\includegraphics[width=2in]{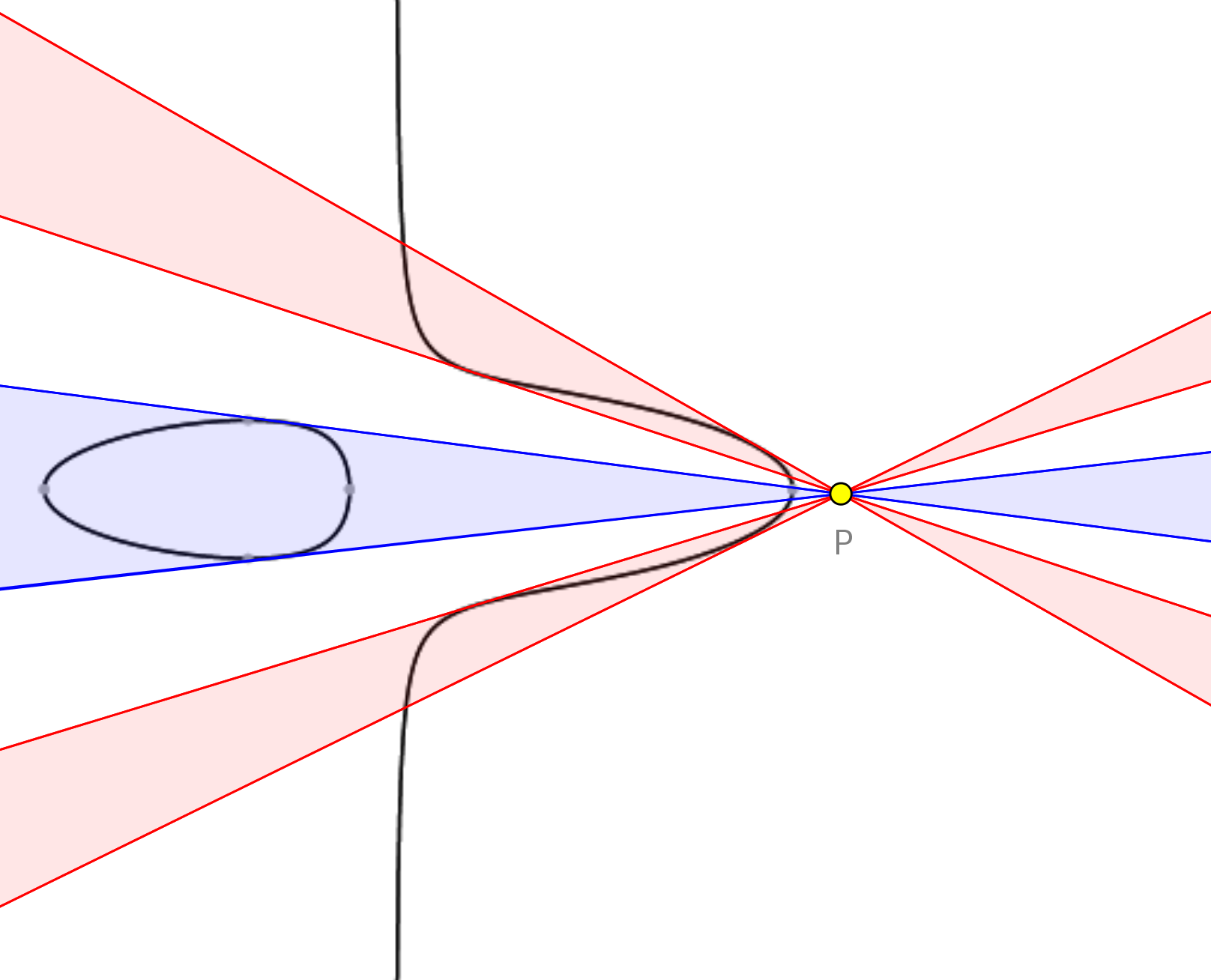}&
\includegraphics[width=3in]{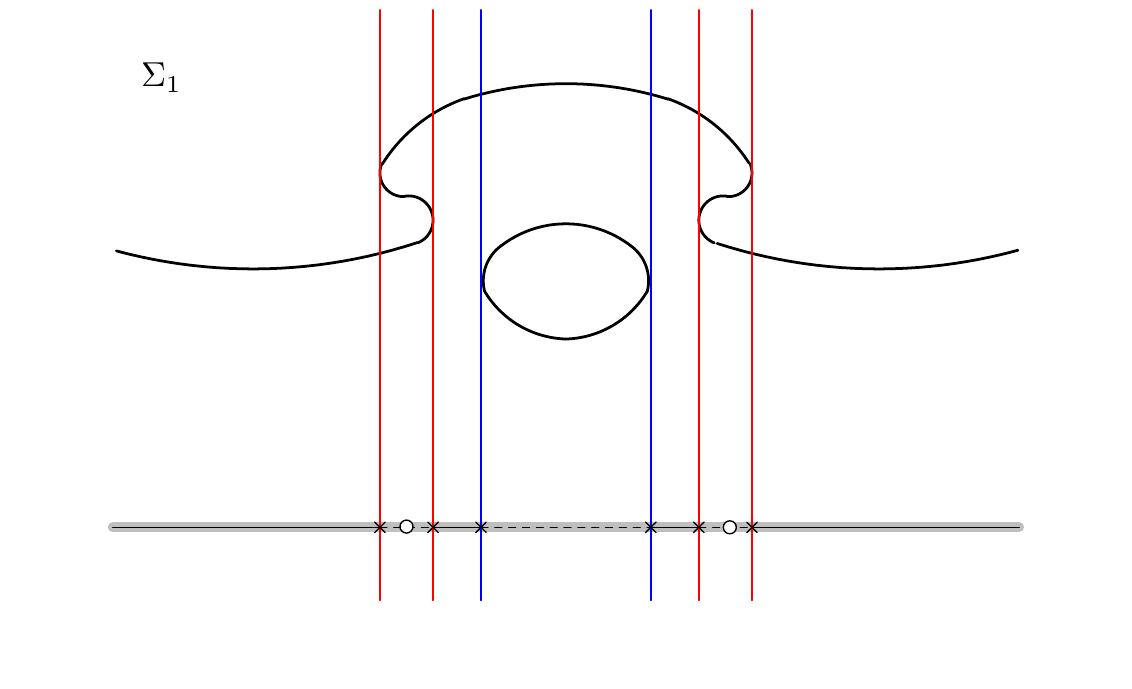}
\end{tabular}
\caption{A real plane cubic with two zigzags and one oval, and its pencil of lines seen in $\Sigma_1$.\label{fig:planeCubic}}
\end{figure}

Let $A\subset\CPP$ be a real smooth cubic. Let $p\in\CPP$ a real point which does not belong to~$A$ and let $B\cong\CP\subset\CPP$ be a real line which does not pass through $p$. The pencil of lines passing through $p$ can be parametrized by $B$, mapping every line $L\ni p$ to $L\cap B$. The blow-up of $\CPP$ at~$p$ is isomorphic to a real geometrically ruled surface over $B$. The strict transform of~$A$ is a real proper trigonal curve $C\subset\Sigma_1$ (since it is proper, it is already a minimal proper model for $A$). Since the real structures are naturally compatible, we associate to $C$ its real dessin $\Dssn(C)_{c}$ on the quotient of $B\cong\CP$ by the complex conjugation. 
Up to elementary equivalence, all the possible dessins are shown on Figure~\ref{fig:cubiques}. They are named by either their type (cf. Definition~\ref{df:type}) and the number of zigzags they possess (in the non-hyperbolic case, cf. Definition~\ref{df:ovzz}) or by $H$ in the case of the \emph{hyperbolic cubic}. 
Up to weak equivalence there are only three classes of cubics, namely the ones of type $\mathrm{I}$, type $\mathrm{II}$ and the hyperbolic one, corresponding to the rigid isotopy classification of couples $(A,p)$ of real cubic curves and one additional point of the plane outside $A$.

\begin{figure} 
\centering
\includegraphics[width=5in]{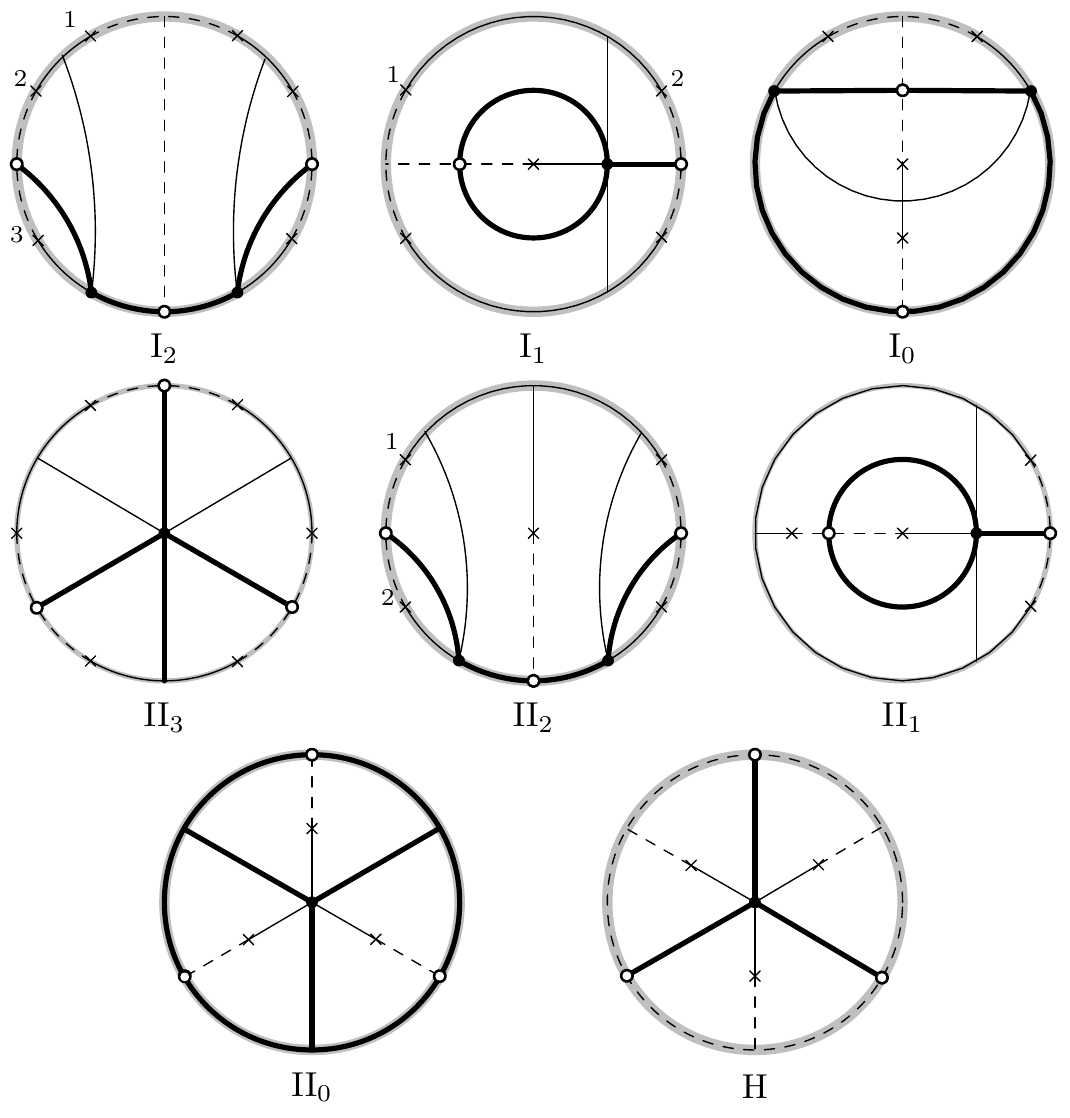}
\caption{Classes of cubic dessins up to elementary equivalence.\label{fig:cubiques}}
\end{figure}
\subsection{Real pointed quartic curves}

Let $A$ be a curve of degree $4$ in $\RPP$, and let~$p\in {\mathbb R} A$ be a point. 
Consider the geometrically ruled surface\linebreak[4] ${\Sigma_1\cong\operatorname{Bl}_{p}(\CPP)\lra\CP}$ which is the blow-up at the point~$p$ of the complex projective plane.
Assume that~$p \in A$ is smooth. 
Then, the strict transform of~${\mathbb C} A\subset\CPP$ is a trigonal curve~$C\subset\Sigma_1$.
The intersection of the exceptional divisor $E$ with the strict transform~$C$ consists of one point.
Geometrically, it represents the tangent line~$T_p A$ to~$A$ at~$p$. Performing the Nagata transformation on the intersection point~$q\in E\cap C\subset\Sigma_1$~transforms~$C$ into a proper trigonal curve~$\widetilde{C}\subset\Sigma_2$. Since the fiber of~$C\subset\Sigma_1\lra\CP$ passing through~$q$ intersects~$C$ in two other points, the trigonal curve~$\widetilde{C}$ has a singular point.

For a curve of even degree in $\RPP$ we say that an oval is  \emph{even} if it lies inside an even number of other ovals. Otherwise, it is called \emph{odd}.
Harnack's innequality states that the number of components of a non-singular curve of degree $d$ in $\RPP$ is at most $\frac{(d-1)(d-2)}{2}+1$.
We say that a curve is \emph{maximal} if it has the maximal number of connected components in $\RPP$.

Consider coordinates such that $\RPP\setminus T_p A \cong \R^2$. 
If there exist a connected component of the image of $A\setminus T_p A$ such that it is contain in the convex hull of the remaining connected components, we say that the components of $A$ are \emph{not in convex position} with respect to $T_p A$.

We are interested in the moduli space of real pointed quartic curves. Therefore, we impose genericity conditions 
on the couple $(A,p)$, namely we assume the curve $A$ to be smooth and we assume the point $p$ to be a point of $A$ such that the tangent line of $A$ at $p$ intersects $A$ in two other distinct points. Equivalently, the tangent line of $A$ at $p$ is not a bitangent of $A$ and $p$ is not an inflection point of $A$. This condition corresponds to the genericity condition for the Weierstraß coefficients sections presented in Definition~\ref{df:gen} within trigonal curves with a   nodal point. 

\begin{prop}[Classification of degree 6 uninodal toiles]\label{prop:deg4}
 The weak equivalence class of a degree 6 uninodal toile is determined by its number of ovals, its type, the nature of the segment containing the only singular $\times$-vertex, and for the type~$\mathrm{I}$ dessins, on the type~$\mathrm{I}$ labeling of the singular $\times$-vertex, as shown in Tables \ref{fig:d4b01} to \ref{fig:d4b04}.
\end{prop}

\begin{proof}
Let $v$ be the nodal {\tv}.
If $v$ is a non-isolated nodal {\tv}, let $S$ be the dotted segment containing $v$ and let $S_1$, $S_2$ be the connected components of~$S\setminus\{v\}$.
Up to weak equivalence, the set of parities of the numbers of white vertices in $S_1$, $S_2$ is an invariant.

Due to Proposition \ref{prop:unide}, a uninodal toile of degree 6 is the gluing of two cubics {\it via} a dotted cut or an axe. 
If the gluing is \emph{via} an axe, the residual dessins are non-singular cubics as in Figure~\ref{fig:cubiques}.
The image of $v$ in a residual cubic is a simple {\tv} $\tilde{v}$.
Up to weak equivalence, a zigzag of the cubic that does not intersect the gluing can be neglected.
Hence, it is enough to consider the cubics $I_1$ and $II_1$ with the {\tvs} marked with $1$ and $2$.
Not every dessin can be obtain as a decomposition as gluing of cubics \emph{via} an axe.
If the gluing is \emph{via} a dotted cut, one residual dessin is non-singular and the other one is uninodal.
A gluing of dessins is of type~$\mathrm{I}$ if and only if its residual dessins are of type~$\mathrm{I}$ and the labelling is coherent at the gluing.
Assembling all possible combinations of dessins in the described manner and grouping them by weak equivalence classes yields to the classification of degree 6 uninodal toiles.
\end{proof}

\begin{thm}[Classification of generic real pointed quartic curves]\label{thm:quart}
 There is a one-to-one correspondence between the weak equivalence classes of degree 6 uninodal toiles and the chambers of generic real pointed quartic curves in their moduli space, given by the figures in Tables~\ref{fig:d4b01} to \ref{fig:d4b04}.
\end{thm}
 
\begin{proof}
This is a consequence of Theorem \ref{th:correpondance2}
and the fact that creating/destroying a zigzag are deformations that remain within the rigid isotopy class.
Every weak equivalence class found in Proposition~\ref{prop:deg4} corresponds to a rigid isotopy class of generic real pointed quartic curves~$(A,p)$. 

A dessin is hyperbolic if and only if every line passing through $p$ intersects~$A$ in four real points, counted with multiplicities. This only occurs if~$\R A$ consist of two nested ovals and~$p$ belongs to the odd oval.
An oval of the dessin that do not containt the nodal {\tv} $v$ corresponds to an oval of the curve $A$ that does not contain $p$ and does not intersect the tangent line $T_p A$.

If the vertex $v$ belongs to a solid segment, the tangent line $T_p A$ intersects $A\setminus\{p\}$ at two complex conjugated points that become a nodal point on the proper trigonal curve associated to $A$ after the contraction in the Nagata transformation.

Otherwise, the vertex $v$ belong to a dotted segment $S$.
When the two connected components of $S\setminus\{v\}$ have an even number of white vertices, the segment $S$ corresponds to an oval of the curve $A$ that do not contain~$p$ and intersects the tangent line $T_p A$.
A connected component of~$S\setminus\{v\}$ with an odd number of white vertices represents a connected component~$\alpha$ of~$A\setminus T_p A$ such that~$p$ belongs to the closure~$\bar{\alpha}$ and for every~$q\in\alpha$ the line passing through~$p$ and~$q$ intersects~$A$ in four real points.

There are two weakly equivalence classes of degree $6$ uninodal toiles with three ovals such that $v$ belongs to a solid segment. Their type~$\mathrm{I}$ labeling of the singular $\times$-vertex is different.
An oval of a dessin can be arbitrarily close to the long component if there exists a zigzag such that they have simple {\tvs} adjacent to a same solid monochrome vertex, or if there exists an inner simple {\tv} adjacent to a dotted monochrome vertex that divides the dotted segment of the oval in two connected components, at least one having an odd number of white vertices.
In the dessin $D$ in Table~\ref{fig:d4b04}, two ovals can be arbitrarily close to the long component, while in the dessin $D'$ in Table~\ref{fig:d4b04} only one oval can be arbitrarily close to the long component.
The long component of such dessins corresponds to the oval of~$A$ containing~$p$.
Hence, the connected components of~$A\setminus T_p A$ are in a convex position for the dessin~$D$, and they are not in a convex position for the dessin~$D'$.
\end{proof}
 
\begin{table}[h]
\caption{$b_0(\R A)=1$.}
\begin{center}
\begin{tabular}{|c|c|c|}
\hline
\includegraphics[width=1.1in]{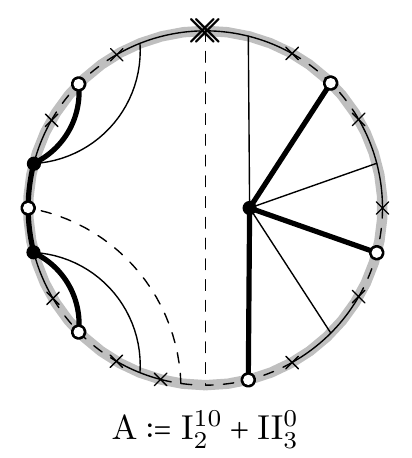}
&
\includegraphics[width=1.1in]{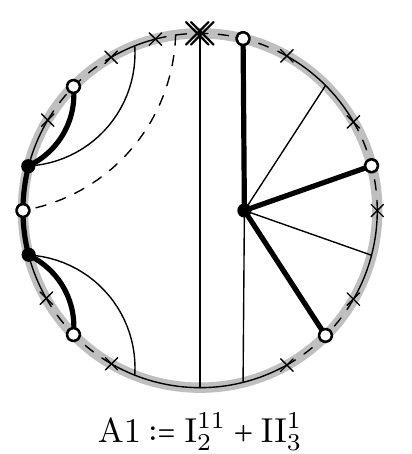}
&
\includegraphics[width=1.1in]{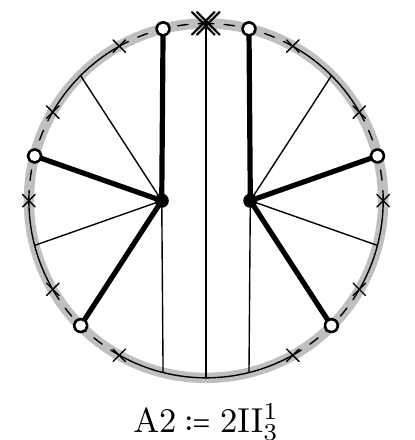}
\\
\includegraphics[width=1.2in]{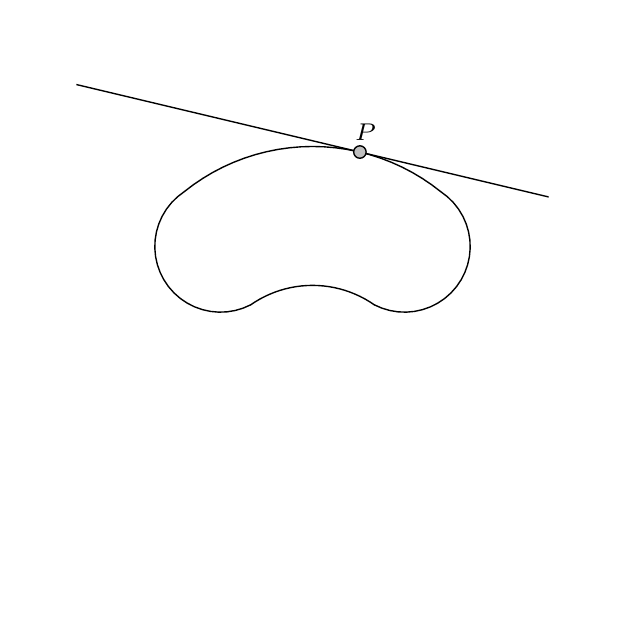}
&
\includegraphics[width=1.2in]{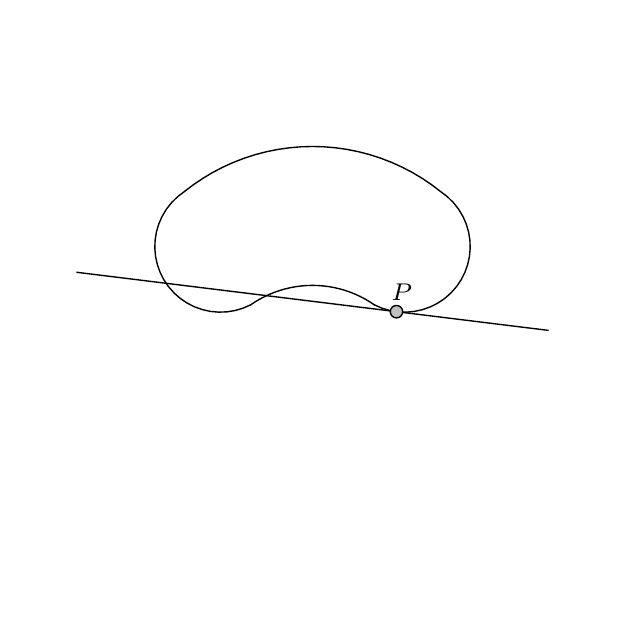}
&
\includegraphics[width=1.2in]{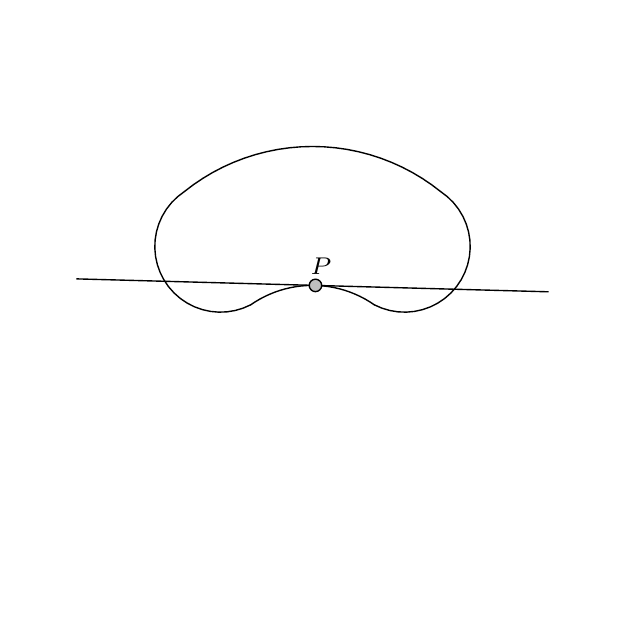}
\\
\hline
\end{tabular}
\end{center}
\label{fig:d4b01}
\end{table}%

\begin{table}[h]
\caption{$b_0(\R A)=2$, type II.}
\begin{center}
\begin{tabular}{|cc|cc|}
\hline
\includegraphics[width=1.1in]{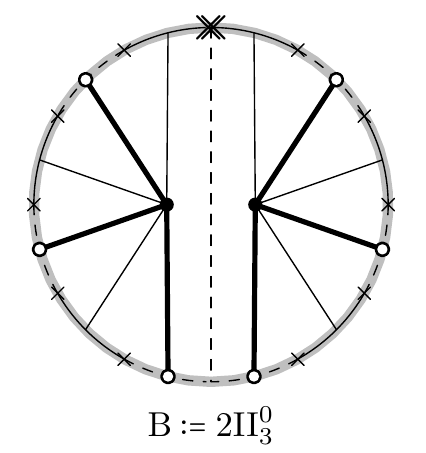}
&
\includegraphics[width=1.2in]{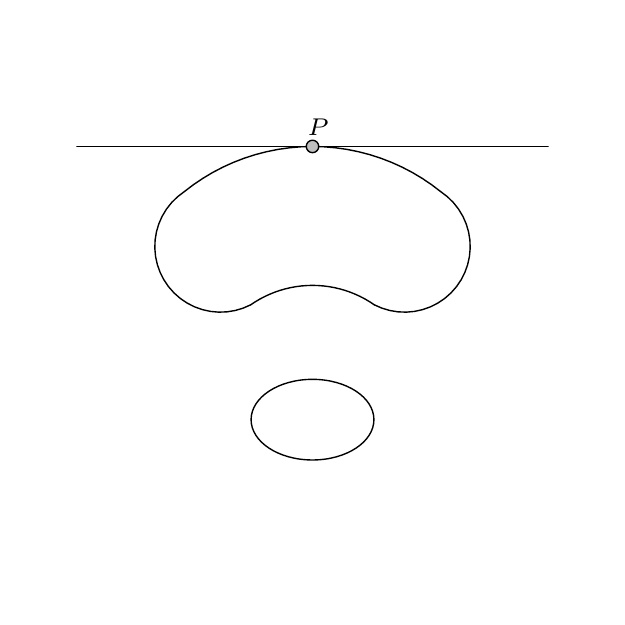}
&
\includegraphics[width=1.1in]{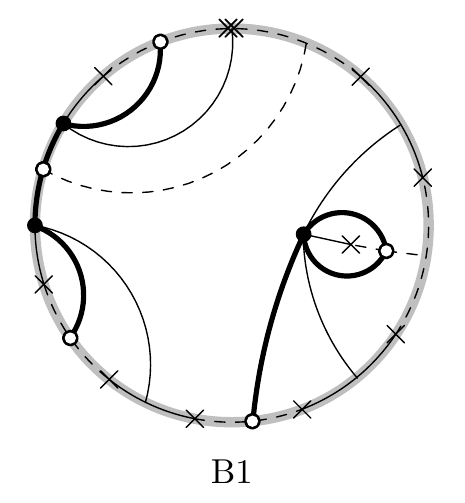}
&
\includegraphics[width=1.2in]{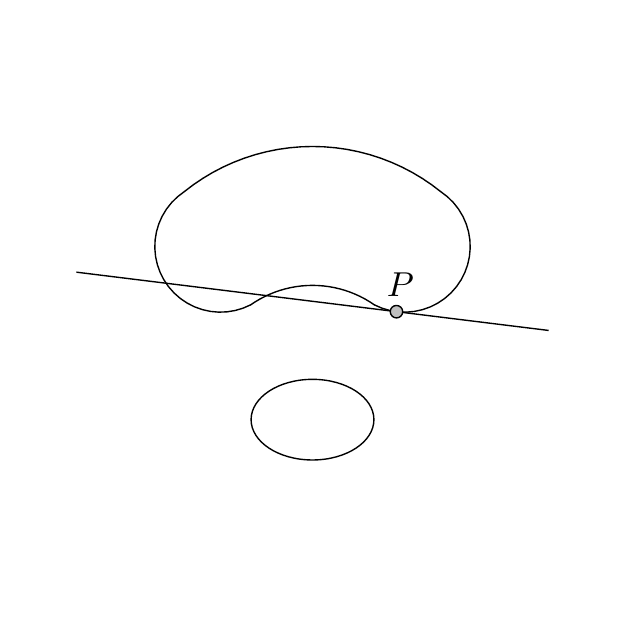}
\\ \hline
\includegraphics[width=1.1in]{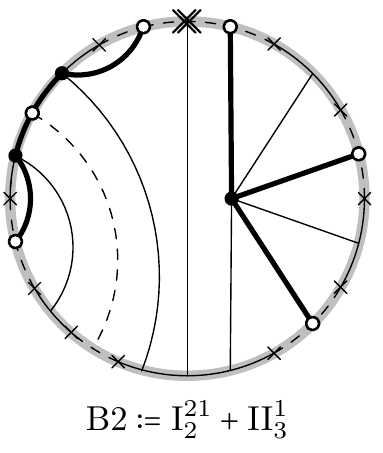}
&
\includegraphics[width=1.2in]{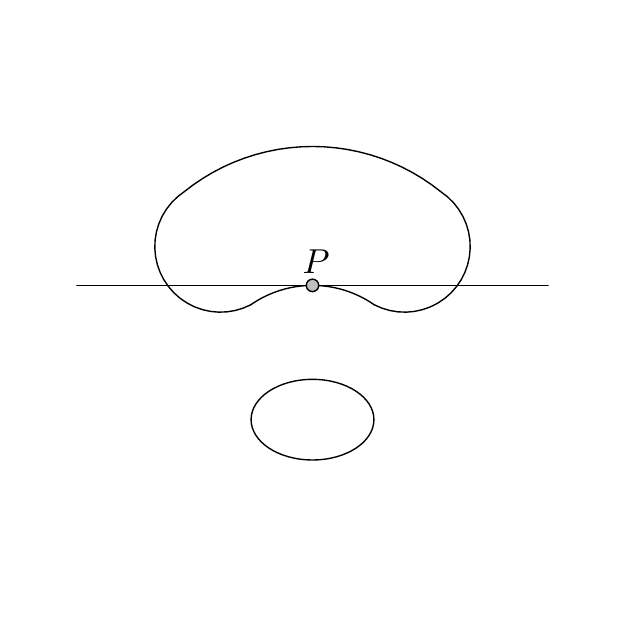}
&
\includegraphics[width=1.1in]{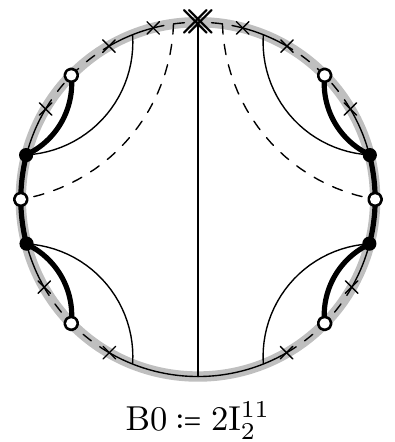}
&
\includegraphics[width=1.2in]{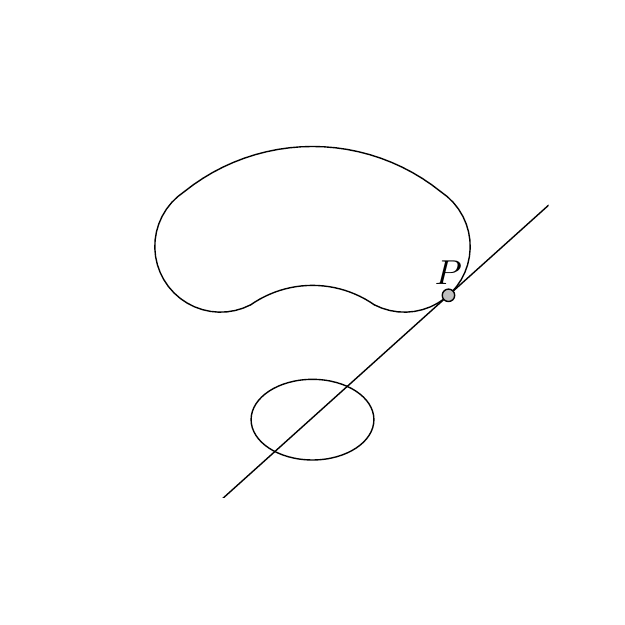}
\\ \hline
\end{tabular}
\end{center}
\label{fig:d4b02II}
\end{table}%

\begin{table}[h]
\caption{$b_0(\R A)=3$.}
\begin{center}
\begin{tabular}{|cc|cc|}
\hline
\includegraphics[width=1.1in]{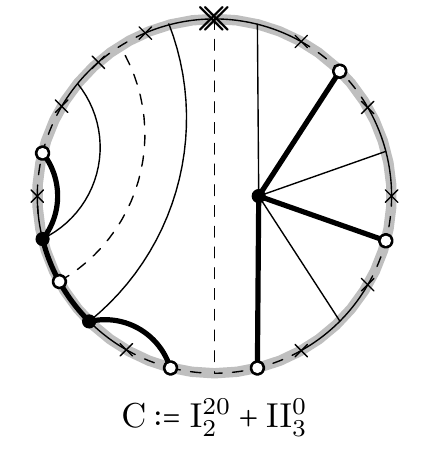}
&
\includegraphics[width=1.2in]{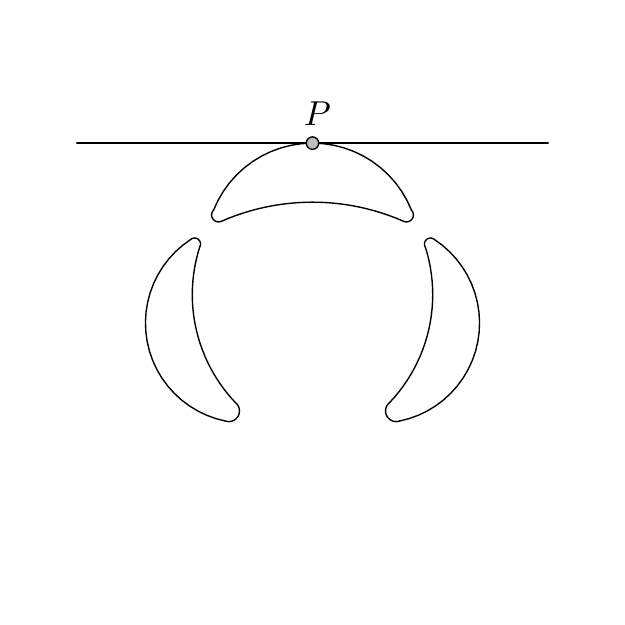}
&
\includegraphics[width=1.1in]{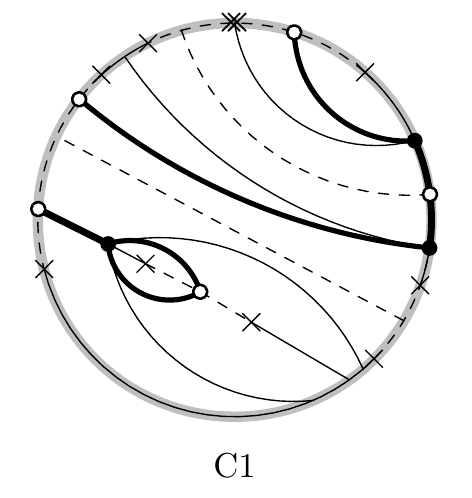}
&
\includegraphics[width=1.2in]{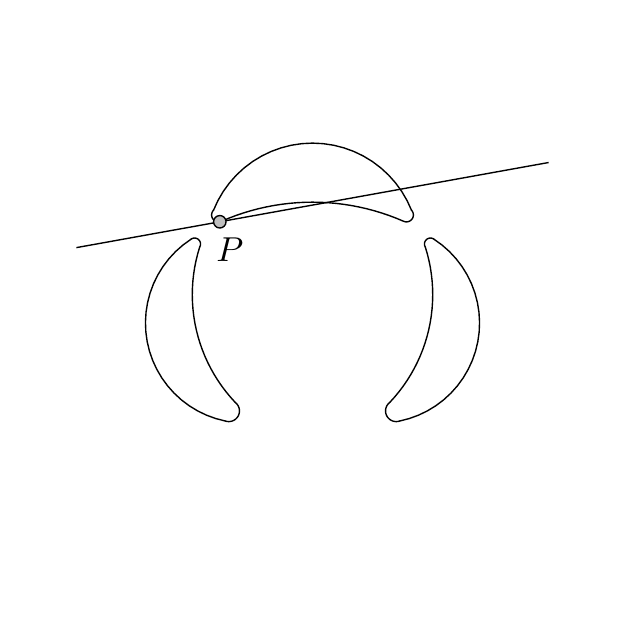}
\\ \hline
\includegraphics[width=1.1in]{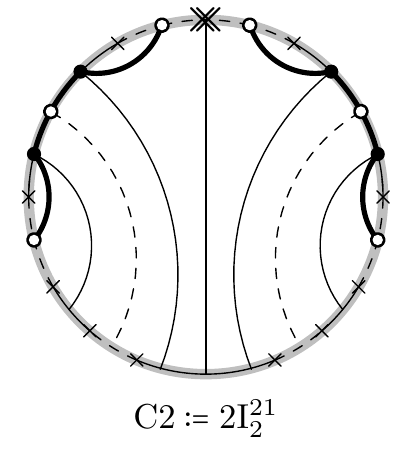}
&
\includegraphics[width=1.2in]{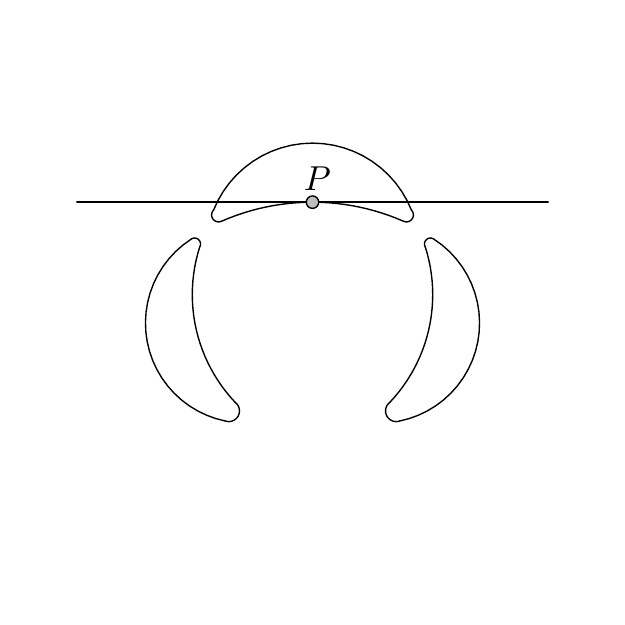}
&
\includegraphics[width=1.1in]{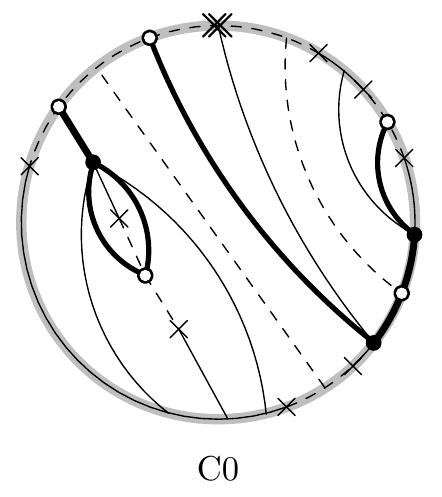}
&
\includegraphics[width=1.2in]{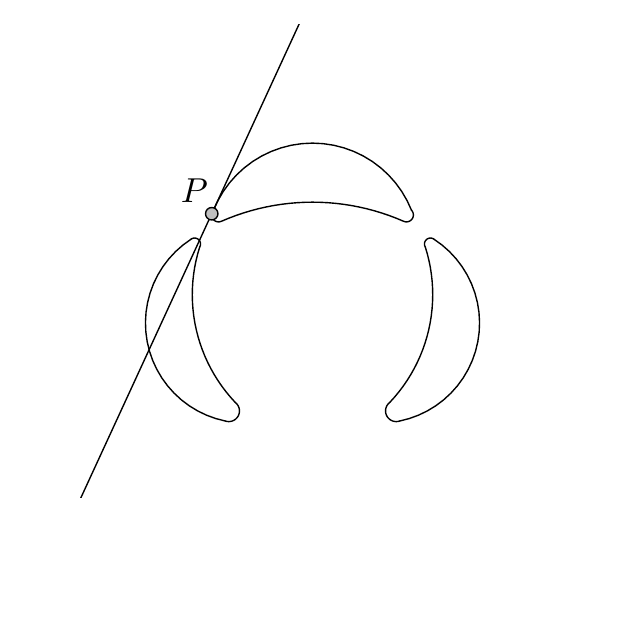}
\\ \hline
\end{tabular}
\end{center}
\label{fig:d4b03}
\end{table}

\begin{table}[h]
\caption{$b_0(\R A)=2$, type I.}
\begin{center}
\begin{tabular}{|cc|cc|}
\hline
\includegraphics[width=1.1in]{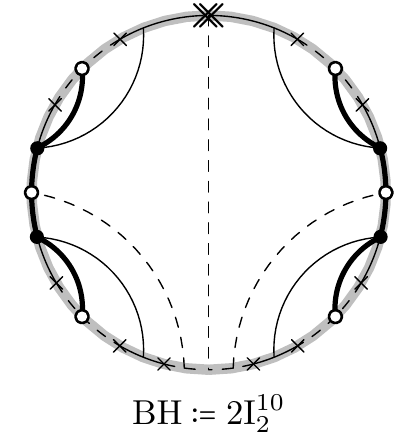}
&
\includegraphics[width=1.2in]{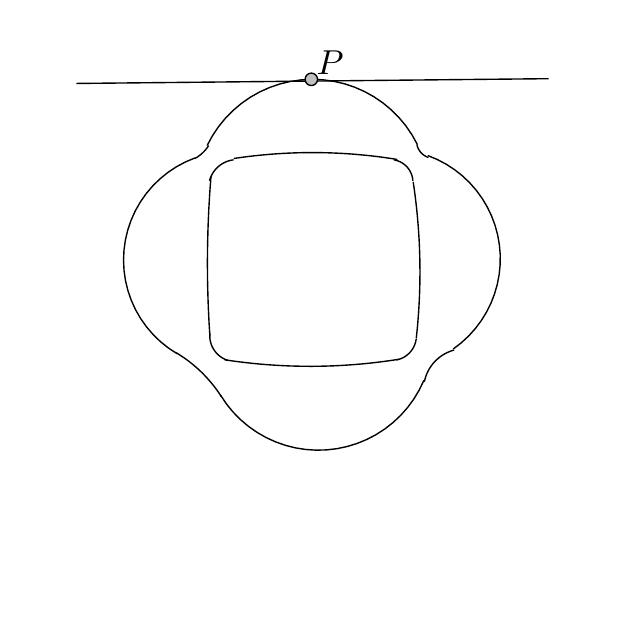}
&
\includegraphics[width=1.1in]{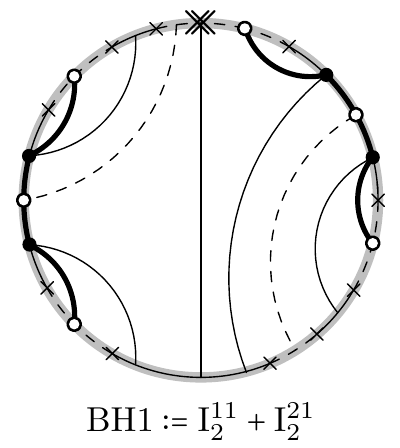}
&
\includegraphics[width=1.2in]{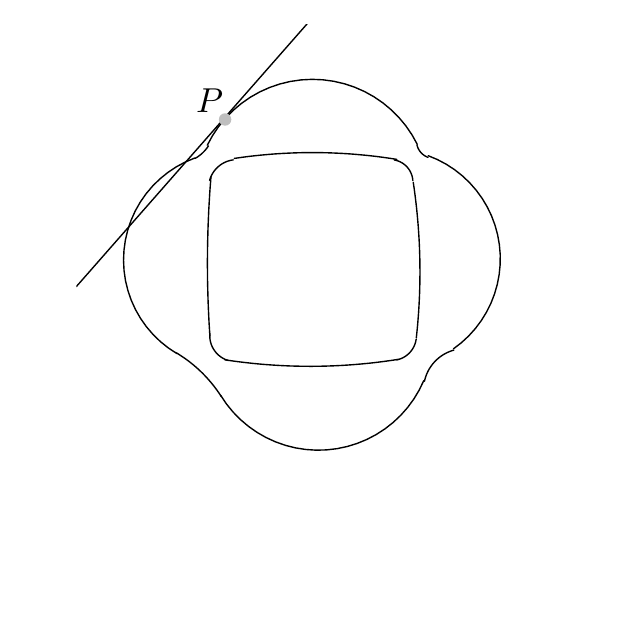}
\\ \hline
\includegraphics[width=1.1in]{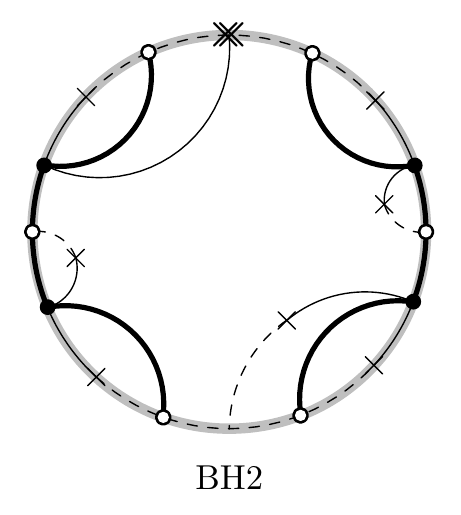}
&
\includegraphics[width=1.2in]{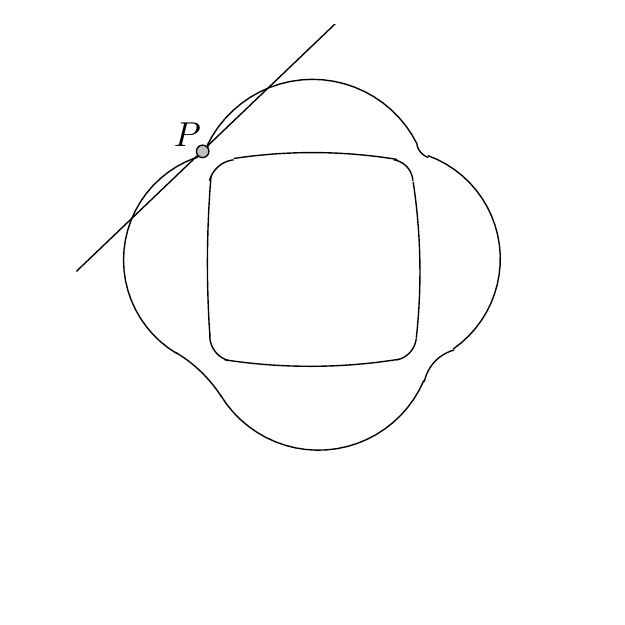}
&
\includegraphics[width=1.1in]{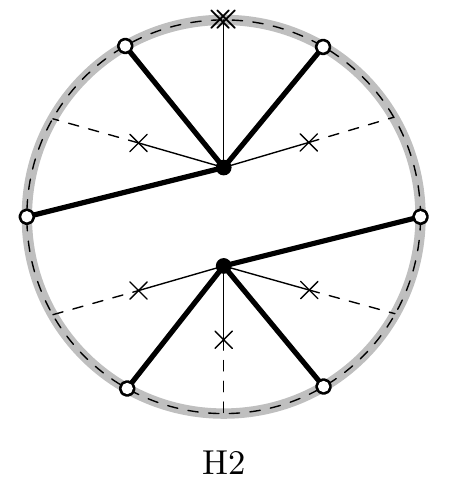}
&
\includegraphics[width=1.2in]{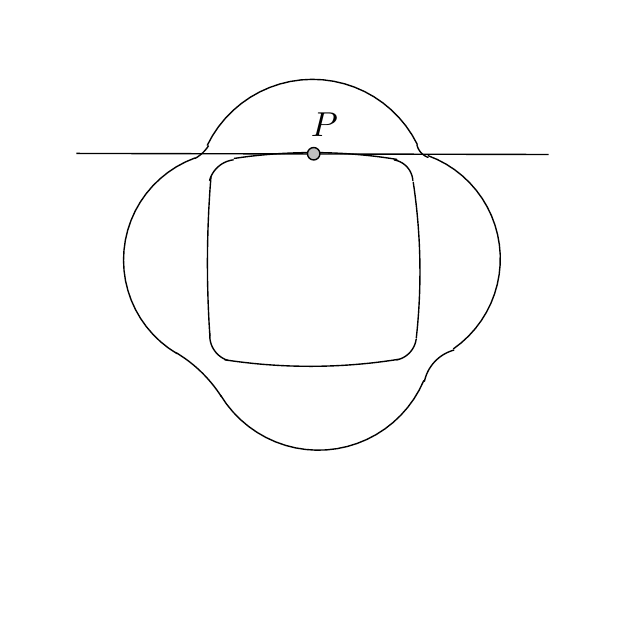}
\\ \hline
\end{tabular}
\end{center}
\label{fig:d4b02I}
\end{table}%

\begin{table}[h]
 \caption{$b_0(\R A)=4$.}
\begin{center}
\begin{tabular}{|c|c|c|}
\hline
\includegraphics[width=1.1in]{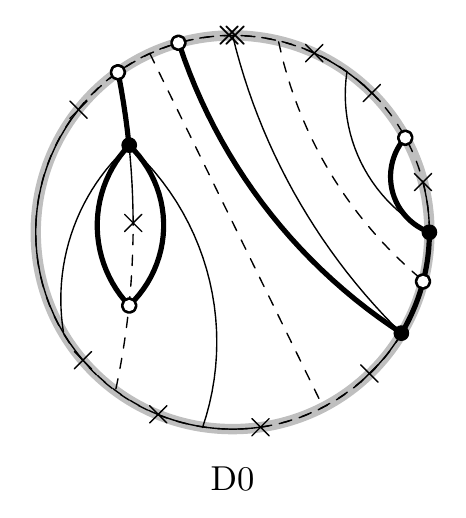}
&
\includegraphics[width=1.1in]{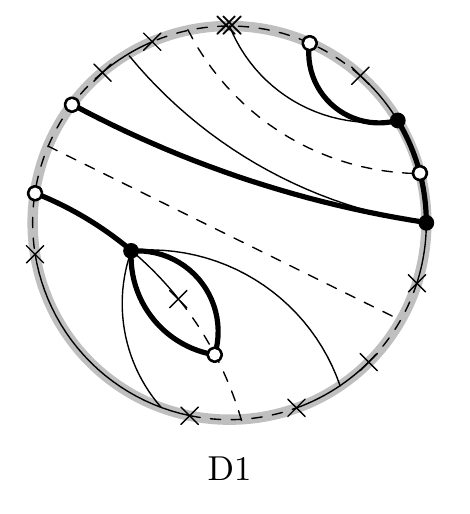}
&
\includegraphics[width=1.1in]{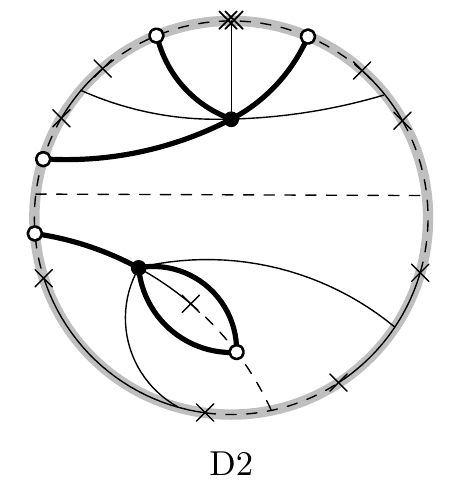}
\\
\includegraphics[width=1.2in]{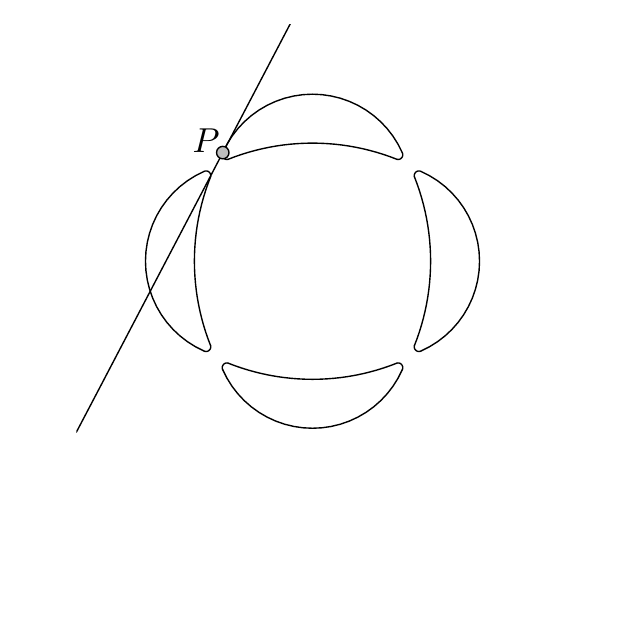}
&
\includegraphics[width=1.2in]{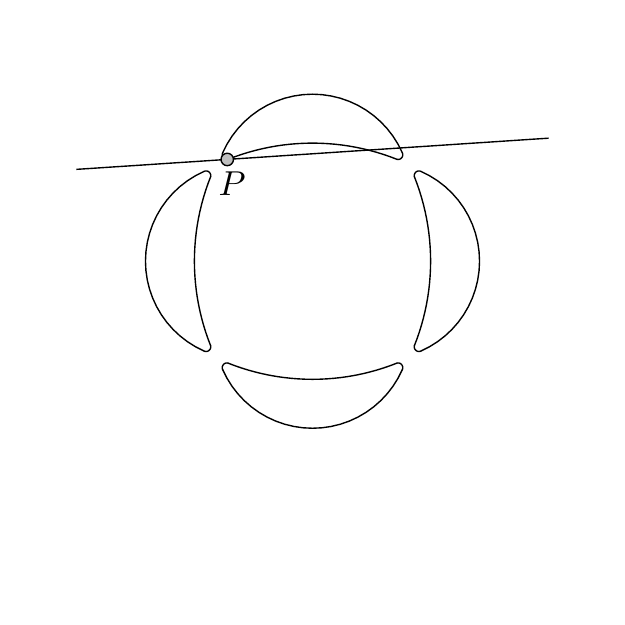}
&
\includegraphics[width=1.2in]{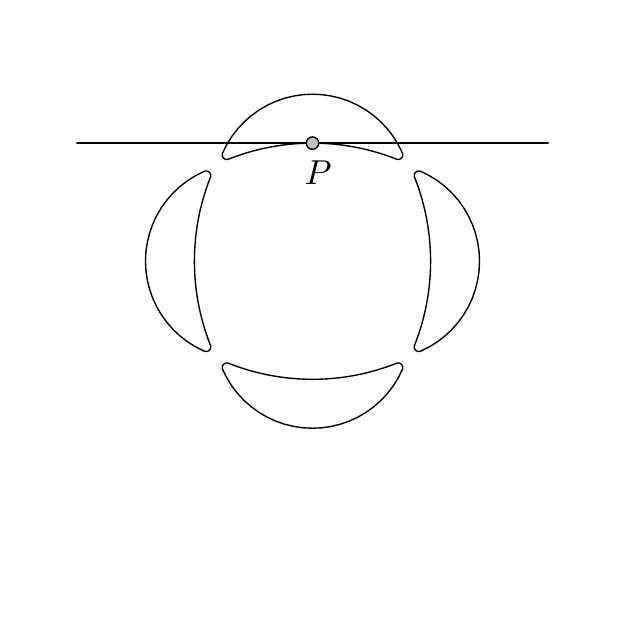}
\\ \hline
\end{tabular}

\begin{tabular}{|cc|cc|}
\hline
\includegraphics[width=1.1in]{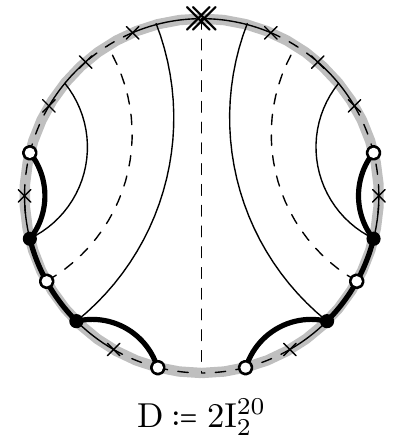}
&
\includegraphics[width=1.2in]{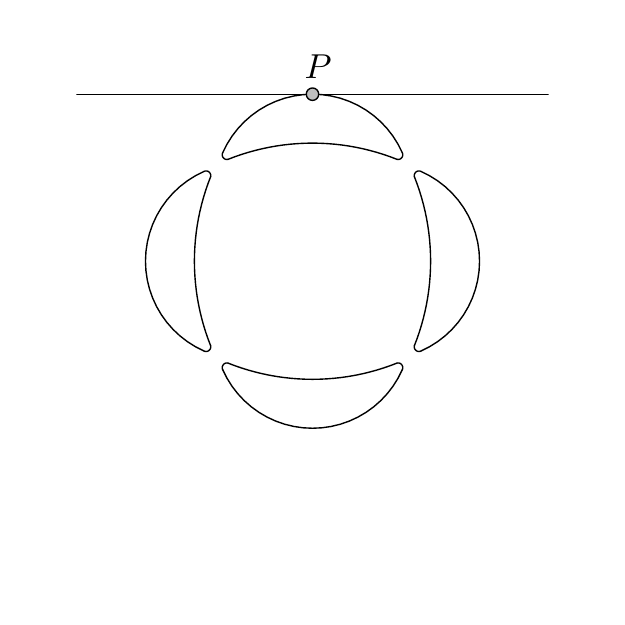}
&
\includegraphics[width=1.1in]{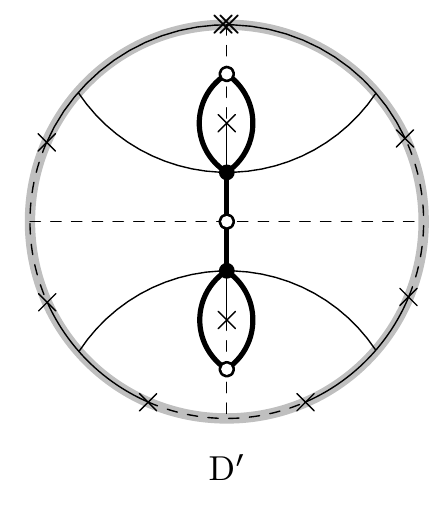}
&
\includegraphics[width=1.2in]{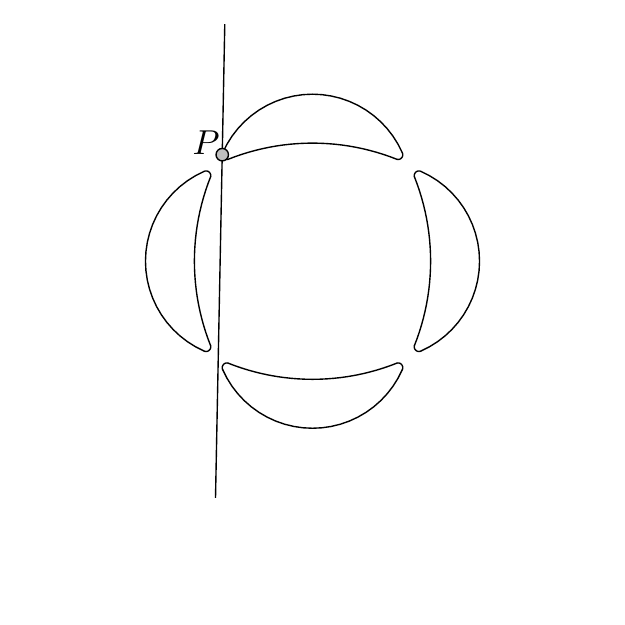}
\\ \hline
\end{tabular}
\end{center}
\label{fig:d4b04}
\end{table}

\bibliographystyle{plain} 
\bibliography{bibliothese} 
\end{document}